\documentclass{article}
\usepackage{amsmath,amsfonts,amsthm,amssymb,amscd,color,xcolor,mathrsfs,eufrak,}
 \usepackage{hyperref}
 \usepackage{comment}

\colorlet{darkblue}{blue!50!black}

\hypersetup{
    colorlinks,%
    citecolor=darkblue,%
    filecolor=red,%
    linkcolor=darkblue,%
    urlcolor=magenta,%
    pdfnewwindow=true,%
    pdfstartview={FitH}
}

\newcommand{\p}{\partial}
\newcommand{\e}{\varepsilon}

\newcommand{\R}{{\mathbb R}}
\newcommand{\Z}{{\mathbb Z}}

\newcommand{\T}{{\mathbb T}}

\newcommand{\N}{{\mathbb N}}
\newcommand{\la}{\lambda}

\newcommand{\ty}{\infty}

\newcommand{\de}{\delta}

\newcommand{\mynegspace}{\hspace{-0.12em}}
\newcommand{\lvvvert}{\rvert\mynegspace\rvert\mynegspace\rvert}
\newcommand{\rvvvert}{\rvert\mynegspace\rvert\mynegspace\rvert}

\newcommand{\aA}{{\cal A}}
\newcommand{\BB}{{\cal B}}

\newcommand{\EE}{{\cal E}}
\newcommand{\FF}{{\cal F}}
\newcommand{\GG}{{\cal G}}
\newcommand{\HH}{{\cal H}}
\newcommand{\KK}{{\cal K}}
\newcommand{\LL}{{\cal L}}

\newcommand{\RR}{{\cal R}}

\newcommand{\WW}{{\cal W}}
\newcommand{\XX}{{\cal X}}

\newcommand{\lag}{\langle}
\newcommand{\rag}{\rangle}

\newcommand{\dd}{{\textup d}}

\newcommand{\diver}{\mathop{\rm div}\nolimits}

\theoremstyle{plain}
\newtheorem*{mtheorem}{Main Theorem}

\newtheorem{theorem}{Theorem}[section]
\newtheorem{lemma}[theorem]{Lemma}
\newtheorem{proposition}[theorem]{Proposition}
\newtheorem{corollary}[theorem]{Corollary}
\theoremstyle{definition}
\newtheorem{definition}[theorem]{Definition}

\theoremstyle{remark}

\numberwithin{equation}{section}

\begin{document}
% \author{  V. Nersesyan\footnote{Laboratoire de Mat\'ematiques, UMR CNRS 8100, Universit\'e de Versailles-Saint-Quentinen-Yvelines, F-78035 Versailles, France;  e-mail: Vahagn.Nersesyan@math.uvsq.fr }}

\title{Approximate controllability of Lagrangian trajectories of  the   3D Navier--Stokes system    by a finite-dimensional   force
}
\date{  }
\author{Vahagn Nersesyan\footnote{Laboratoire de
Math\'ematiques, UMR CNRS 8100, Universit\'e de
Versailles-Saint-Quentin-en-Yvelines, F-78035 Versailles,
France}
}
  \maketitle

\begin{abstract} 

In the Eulerian   approach, the motion of an  incompressible   fluid is usually described by the velocity field which is  given by the  Navier--Stokes system.   The velocity field generates a   flow in the space of   volume-preserving diffeomorphisms. The latter     plays a central role in the Lagrangian description of a fluid, since it  allows to    identify  the trajectories of the  individual particles.  In this paper, we show that 
   the velocity field of the fluid and the corresponding  flow of  diffeomorphisms  
      can be simultaneously approximately controlled using a finite-dimensional external force.  
The proof is based on some methods from the geometric control theory introduced by Agrachev and Sarychev.


 \end{abstract}

  \tableofcontents
\setcounter{section}{-1}

\section{Introduction}
\label{s0}
The motion of an   incompressible fluid     is described by the following Navier--Stokes (NS) system 
 \begin{align}
\p_t u-\nu \Delta u+\lag u, \nabla \rag u+\nabla p  &= f(t,x), 
\quad  \diver u =0,   \label{0.1}\\
u(0)&=u_0,\label{ic}
\end{align} 
where $\nu>0$ is the kinematic viscosity, ${  u} = (u_1(t,x), u_2(t,x), u_3(t,x))$ is the velocity field of the fluid, $p = p(t, x)$ is the pressure, and~$f$ is  an   external force. Throughout this paper,  we shall assume that  the space       variable  $x=(x_1,x_2,x_3)$   belongs  to the torus~$\T^3= \R^3/2\pi\Z^3$. 

\smallskip
The well-posedness of  the  3D NS system \eqref{0.1} is a famous open problem.
Given   smooth data $(u_0,f)$, the existence and uniqueness  of a smooth solution is known to hold only locally in time.  The global existence  is established in the case of  small data.      For  large data the global existence holds in the case of a weak solution, but in that case the uniqueness is   open.

\smallskip
The flow generated by a sufficiently smooth velocity field $u$
    gives  the  Lagrangian 
trajectories of the fluid:
 \begin{equation}\label{0.3}
\dot x=u(t,x),  \quad x(0) =x_0 \in \T^3.   
\end{equation} 
Since the fluid is assumed to be incompressible, for any   $t\ge0$,   the mapping $\phi^u_t: x_0\mapsto x(t)$ belongs to the group  $\textup{SDiff}(\T^3)$ of orientation and volume preserving diffeomorphisms on $\T^3$ isotopic to the identity. This group is often referred   as {\it configuration space} of the
fluid  (cf.~\cite{ArnoldKh,KeWe-09}). Thus for   sufficiently smooth data, we have a path $(u(t), \phi^u_t)$,   which is   defined locally in time.   The main issue addressed in  this paper is the approximate controllability of the couple  $(u(T), \phi^u_T)$ for any $T>0$.   
We shall   assume that the external force is  of the following  form
$$
f(t,x)=h(t,x)+\eta(t,x),
$$where 
$h$ is the fixed part of the     force (given function) and $\eta$ is a control force. To state the main result of this paper, we need to introduce some notation.
Let  us define the space 
\begin{equation}\label{Hdef}
H:=\{u\in L^2(\T^3, \R^3): \diver u=0,\quad \int_{\T^3}u(x)\dd x=0\},
\end{equation} and denote by  $\Pi$  the orthogonal projection    onto $H$ in $L^2(\T^3,\R^3)$. 
    Consider the projection  of system  \eqref{0.1} onto $H$: 
\begin{equation}\label{Elav}
\dot  u +   Lu + B(u) =h(t,x)+\eta(t,x),
\end{equation}where $L=-\Delta$ is the Stokes operator and $B(u):=\Pi (\lag u, \nabla \rag u)$.  Let us set $H^k_\sigma :=H^k(\T^3,\R^3)\cap H$, where~$H^k(\T^3,\R^3)$ is the space of vector functions $v =
(v_1, v_2 , v_3)$ with components in
   the usual   Sobolev space of order $k$   on~$\T^3$.   
    Let~$E$ a be subset of~$H$. 
   We shall say that system~\eqref{Elav} is approximately controllable by an $E$-valued control, if for any $\nu>0$,   $k\ge3$,   $\e>0$,   $T>0$,     $u_0, u_1\in H^k_\sigma$,     $h\in L^2(J_T,H^{k-1}_\sigma)$, and     $\psi\in \textup{SDiff}(\T^3)$,  there is a control $\eta\in L^2([0,T], E)$   and  a solution $u$ of \eqref{Elav}, \eqref{ic}   defined for any $t\in [0,T]$ and  satisfying 
$$
\|u(T) - u_1 \|_{H^k(\T^3)}+ \| \phi_T^u - \psi \|_{  C^1(\T^3)}<\e.
$$
  The following theorem is a simplified
version of our main result (see Section \ref{s2},  Corollary~\ref{D.3.1}). 
\begin{mtheorem}
There is a finite-dimensional subspace $E\subset H $ such that~\eqref{Elav} is approximately controllable by  an $E$-valued control.  \end{mtheorem} 
Roughly speaking, this    shows that, using a finite-dimensional external force, one can drive the fluid flow (which starts at the identity)     arbitrarily close to any   configuration~$\psi\in \textup{SDiff}(\T^3)$. 
 Moreover, near the final position $\psi(x)$,  the particle starting from $x$ will have approximately  the prescribed velocity~$v_1(x):=u_1(\psi(x))$.  Note that  $\phi_T^u$ depends not only on $u(T)$, but on the whole path~$u(t), t\in[0,T]$.  
 Thus one  needs  a   {\it path-controllability property} for   the velocity field  in order to prove    controllability for  $\phi_T^u$. This path-controllability is one of the novelties of this paper, it    is established in 
  Theorem~\ref{T.2.1}.

\smallskip

We   give some explicit examples of finite-dimensional  subspaces $E$ which ensure the above approximate controllability property.    For instance, for any~$\ell\in \Z^3$,  let $\{l(\ell),l(-\ell)\}$  be an arbitrary orthonormal basis in   
$\{x\in \R^3: \langle x, \ell \rangle=0\}$. We show that our problem is controllable 
by $\eta$ taking values in  a space      of the form 
\begin{equation}\label{sate}
E= E(\KK):= \textup{span}\{ l(\pm\ell) \cos\langle \ell, x \rangle,   l(\pm\ell)
\sin\langle \ell, x \rangle: \ell \in\KK \}, \,\,\KK \subset \Z^3
\end{equation}
 if and only if  $\KK$ is a generator of $\Z^3$ (i.e., any $a\in \Z^3$ is a finite linear combination of the elements of $\KK$ with integer coefficients).    The simplest example  of a generator of~$\Z^3$  is  $$\KK=\{(1,0,0), (0,1,0), (0,0,1)\},$$ in which case~$\dim E(\KK)=12$. 
 We also establish approximate controllability   of the system in question  by controls    having two vanishing components. More precisely,   the space $E$ can be    chosen of the form 
 \begin{equation}\label{CLhod}
 E=\Pi\{(0,0,1) \zeta: \zeta \in \HH\},\end{equation} where 
$$
\HH:= \textup{span}\{ \sin\lag m,x\rag,  \cos\lag m,x\rag: m\in\KK \}
$$  and   $  \KK:=\{(1,0,0), (0,1,0), (1,0,1), (0,1,1) \}
$ (i.e., $\dim E=8$). In \eqref{tilde}   an example of a 6-dimensional subspace is given which guarantees the controllability of the 3D NS system.

\smallskip

The strategy of the proof  of Main Theorem is  based on some   methods   introduced by Agrachev and Sarychev in \cite{AS-2005, AS-2006}  (see also the survey~\cite{AS-2008}).  In that papers they prove  approximate controllability  for the  2D NS  and   Euler systems  
by a finite-dimensional force. This method is then developed  and generalised by  several authors for  various   PDE's.     Rodrigues   proves in~\cite{SSRodrig-06}  controllability for the 2D NS system  on a  rectangle with the Lions boundary
conditions, and in~\cite{Rodrig-2006, Rodrig-2008} he extends the results to the case of more general Navier boundary conditions and the Hemisphere under the   Lions boundary
conditions.  The controllability for the   3D NS     system  on the   torus is studied in \cite{shirikyan-cmp2006, shirikyan-aihp2007} by Shirikyan. He also considers   the case of the Burgers equation on the real line in  \cite{Shi-2013} and  on an interval with the Dirichlet boundary conditions  in~\cite{shirikyan-aihp2007b, shirikyan-2010}. 
  Incompressible  and compressible    3D Euler equations are considered by Nersisyan in~\cite{Hayk-2010,   Hayk-2011}, and the controllability for the  2D defocusing cubic Schr\"odinger equation is established  by  Sarychev in \cite{Sar-2012}.    In~\cite{shirikyan-2008a} Shirikyan proves that the Euler equations are not exactly controllable by a finite-dimensional
external force.

\smallskip
All the above papers are concerned with the problem of controllability of the velocity field.  
The controllability of the  Lagrangian  trajectories  of    2D and 3D Euler equations is studied  by Glass and Horsin   \cite{GH-2010,GH-2012}, in the case of boundary controls. For given
two smooth contractible sets $\gamma_1$ and $\gamma_2$ of fluid particles which surround the same
volume, they construct a control such that the corresponding flow drives~$\gamma_1$ arbitrarily close to $\gamma_2$.   
In the context of our paper,    a similar property can be derived from our main result. Indeed,   Krygin shows in~\cite{MR0368067}  that  there is a diffeomorphism  $\psi\in \textup{SDiff}(\T^3)$ such that  $\psi(\gamma_1)=\gamma_2  $. Thus we can find an $E$-valued control $\eta$ such that $\phi_T^u(\gamma_1)$ is arbitrarily close to~$ \gamma_2  $, and, moreover, at time $T$ the particles will have approximately the desired velocity.  

When $E$ is of the form \eqref{CLhod}, our Main Theorem is related to   the recent paper \cite{CL-2012} by Coron and Lissy.  In that paper,  the authors establish      local null controllability of the velocity     for  the 3D NS system  controlled by a distributed force    having two vanishing components (i.e., the controls are valued in a space of the form \eqref{CLhod}, where $\HH$ is the space of space-time $L^2$-functions supported in a given open subset). The reader is referred to the book \cite{MR2302744}   for an introduction to the control theory of    the NS system by distributed controls and for  references on that topic.

\smallskip

  Let us give a   brief (and not completely   accurate) description of how the Agrachev--Sarychev method is adapted to establish  approximate controllability in the above-defined sense. We assume that       $E$ is  given by \eqref{sate} for some generator~$\KK$ of $\Z^3$.     
Let   $\psi\in \textup{SDiff}(\T^3)$ and let $I(t,x)$ be a smooth isotopy   connecting it to the identity: $I(0,x)=x$ and $I(T,x)=\psi(x)$. Then  $\hat u(t,x):=\p_t I(t,I^{-1}(t,x))$ is a divergence-free vector field such that $\phi^{\hat u}_t(x)=I(t,x)$ for all $t\in [0,T]$. In particular, $\phi^{\hat u}_T=\psi$.  The mapping $u\mapsto \phi^u_T$ is continuous from  $L^1([0,T], H^k_\sigma)$ to~$ C^1(\T^3)$, where $L^1([0,T], H^k_\sigma)$ is endowed with the   {\it relaxation norm}   
   $$
   \lvvvert u \rvvvert_{T,k}:= \sup_{t\in [0,T]}  \left\|\int_0^t u(s)\dd s\right\|_{H^k(\T^3)}.
   $$
Hence  we can choose a smooth   vector field  $  u$  sufficiently close to  $\hat u$ with respect to this  norm, so    that
$$ u(0)=u_0, \quad   u(T)=u_1, \quad 
  \|\phi^u_T-\psi\|_{C^1(\T^3)} < \e. 
$$Then $u$ is a   solution of  our system corresponding to a  control   $\eta_0$, which can be explicitly expressed in terms of $u$ and $h$ from   equation~\eqref{Elav}. In general, this control $\eta_0$ is not $E$-valued, so we need to  approximate $u$ appropriately  with solutions corresponding to $E$-valued controls. To this end, we define the sets 
 $$
 \KK_0:=\KK, \quad \KK_j=\KK_{j-1}\cup\{m\pm n: m,n\in \KK_{j-1}\}, \quad j\ge1.
 $$As $\KK$ is a generator of $\Z^3$, one easily gets that $\cup_{j\ge1}\KK_j=\Z^3$, hence     $\cup_{j\ge1} E(\KK_j)$ is dense in $H^k_\sigma$. Let $P_N$ be the orthogonal projection onto $E(\KK_N)$ in $H$.    Then a perturbative result implies that, for a sufficiently large $N\ge 1$, system \eqref{Elav}, \eqref{ic}   with   control $P_N \eta_0$ has a   strong solution $u_N$  verifying   (see Theorem \ref{T:pert} and Lemma \ref{L:1.1})
$$
    \|u_N(T)-u_1\|_{H^k(\T^3)}+
 \|\phi^{u_N}_T-\psi\|_{C^1(\T^3)} < \e.
$$ 
On the other hand, if we 
consider   the following auxiliary  system
\begin{align}
\dot{u}+\nu L (u+\zeta)+B (u+\zeta) =h+\eta \label{dsdjk}
\end{align} with two controls $\zeta$ and $\eta$, then the below two properties hold true  
\begin{description}
\item[Convexification principle.] For any $\e>0$ and any solution $u_j$ of     \eqref{Elav}, \eqref{ic} with   an  $E(\KK_j)$-valued control $\eta_1$, there are $E(\KK_{j-1})$-valued controls $\zeta$ and~$\eta$ and a solution $\tilde u_{j-1}$ of \eqref{dsdjk}, \eqref{ic} such that 
 $$    \|u_j(T)-\tilde u_{j-1}(T)\|_{H^k(\T^3)}+
    \lvvvert  u_j-\tilde u_{j-1} \rvvvert_{T,k } < \e. 
$$
\item[Extension principle.]
For any $\e>0$ and any solution $\tilde u _j$ of    \eqref{dsdjk}, \eqref{ic}    with   $E(\KK_j)$-valued controls $\zeta$ and $\eta$, there is an $E(\KK_{j})$-valued control   $\eta_2$ and a solution $u_j$ of  \eqref{Elav}, \eqref{ic}   such that   
$$    \|u_j(T)-\tilde u_{j}(T)\|_{H^k(\T^3)}+ 
    \lvvvert  u_j-\tilde u_{j} \rvvvert_{T,k} < \e.
$$\end{description}

These two principles and  the
  above-mentioned continuity property  of $\phi_T^u$ with respect to the relaxation norm  imply that,    for   any solution~$u_j$ of     \eqref{Elav}, \eqref{ic} with   an  $E(\KK_j)$-valued control $\eta_1$, there is   an  $E(\KK_{j-1})$-valued control $\eta_2$ and a solution $u_{j-1}$ of  \eqref{Elav}, \eqref{ic}   such that   
$$    \|u_j(T)-\tilde u_{j-1}(T)\|_{H^k(\T^3)}+
     \|\phi_T^{  u_j}-\phi_T^{  u_{j-1}}\|_{C^1(\T^3)}< \e.
$$
Combining this  with the above-constructed solution $u_N$, we get   the    approximate controllability of \eqref{Elav}   by  a control valued in    $E(\KK)=E$.  The proofs of convexification and extension principles are   strongly  inspired by~\cite{shirikyan-cmp2006}.
 
 \subsection*{Acknowledgments}  
The author would like to thank the referee for his careful reading     and   valuable comments. This research   was supported by the ANR grants EMAQS (No~ANR  2011 BS01 017 01) and STOSYMAP (No~ANR 2011 BS01 015~01).

 \subsection*{Notation} 
We denote by $\T^d$  the standard $d$-dimensional torus~$\R^d/2\pi\Z^d$. It is endowed with the metric and  the  measure
induced by the usual Euclidean metric and the Lebesgue measure on $\R^d$. More precisely, if    ${ \mathsf{\Pi}}:\R^d\to \T^d$ denotes the canonical projection, we have 
\begin{align*}
d(x,y)&=\inf\{ |\tilde x-\tilde y| : \mathsf{\Pi}\tilde  x= x, \mathsf{\Pi}\tilde  y= y, \tilde x,\tilde y\in \R^d \}\quad \text{for any $x,y\in \T^d$},\\
\dd(A)&= (2\pi)^{-d} \dd_{\R^d}(\mathsf{\Pi}^{-1}(A)\cap [0,2\pi]^d) \quad \text{for any Borel subset   $A\subset \T^d$}, 
\end{align*}where $|x|=|x_1|+\cdots+ |x_d|, x=(x_1, \ldots, x_d)\in \R^d$ and $\dd_{\R^d}$ is the  Lebesgue measure on $\R^d$. We denote by $\lag \cdot,\cdot\rag$ the scalar product in $\R^d$. 

\smallskip
\noindent
   $L^p(\T^d, \R^d)$  and $H^s(\T^d, \R^d)$ stand for     spaces of vector functions $u =
(u_1, \ldots , u_d)$ with components in
   the usual Lebesgue and Sobolev spaces on~$\T^d$.

\smallskip
\noindent
$C^{k,\la}(\T^d, \R^d)$, 
   $k\ge0$,  $\la\in (0,1]$ is   the  
space of       vector functions $u =
(u_1, \ldots, u_d)$ with components     that are continuous on~$\T^d$ together with their derivatives up to order   $k$, and whose
derivatives of order $k$ are H\"older-continuous of exponent~$\la$,  equipped with the norm
$$
\|u\|_{C^{k,\la}}:= \sum_{|\alpha|\le k} \sup_{x\in\T^d}|D^\alpha u(x)|+\sum_{|\alpha|= k} \sup_{x,y\in\T^d,x\neq y}\frac{|D^\alpha u(x)-D^\alpha u(y)|}{d (x,y)^\la}.
$$ % Recall that, by the Morrey theorem,  if $k+\la+d/2\le n$, then   $H^n\subset C^{k,\la}$ (e.g., see \cite{adams1975}).  

\smallskip
\noindent
$H^k_\sigma(\T^d,\R^d): =
H^k(\T^d, \R^d)\cap H$ and~$C^{k,\la}_\sigma (\T^d,\R^d):=C^{k,\la}(\T^d,\R^d)\cap H$, where~$H$ is given by \eqref{Hdef} (with $d$ instead of $3$).
%We   denote   by $  \Pi$ the orthogonal projection  onto $H$ in $L^2(\T^d,\R^d)$. The Stokes operator is defined  by $L := -  \Delta$, $D(L)=H_\sigma^2(\T^d,\R^d)$.
 In what follows,  when the space dimension~$d$ is $3$, we shall   write $L^p,  H^k,   \ldots $ instead of $L^p(\T^3,\R^3),  H^k (\T^3,\R^3) ,  \ldots $.

\smallskip
\noindent
$C^1(\T^d)$ is   the space of continuously differentiable maps from $ \T^d$ to $\T^d$ endowed with the usual distance
$
\|\psi_1-\psi_2\|_{C^1(\T^d)}, $ $  \psi_1,\psi_2\in C^1(\T^d).
$ 

\medskip

Let  $X$ be a Banach  space endowed with a
norm $\|\cdot\|_X$ and $J_T:=[0, T ]$. For $1\leq p<\infty$,  let   $L^p(J_T,X)$ be the
space of measurable functions $u: J_T \rightarrow X$ such that
\begin{equation}
\|u\|_{L^p(J_T,X)}:=\bigg(\int_{0}^T \|u(s)\|_X^p\dd s
\bigg)^{\frac{1}{p}}<\infty.\nonumber
\end{equation}
   The spaces   $C(J_T,X)$ and $W^{k,p}(J_T,X)$ are defined in a similar way. We define the {\it relaxation norm} on $L^1(J_T,X)$ by  
\begin{equation}\label{rexaxX}
   \lvvvert u \rvvvert_{T,X}:= \sup_{t\in J_T}  \left\|\int_0^t u(s)\dd s\right\|_{X}.
   \end{equation}

\smallskip
 
A mapping   $\psi:\T^d\to\T^d$ is volume-preserving if   $\dd (\psi^{-1}(A))= \dd(A)$ for any Borel subset   $A\subset \T^d$. We shall say that $\psi\in C^1(\T^d)$ is orientation-preserving if the differential $D_x\psi$ is an orientation-preserving linear map for   all $x\in \T^d$.
We denote by  $\textup{SDiff}(\T^d)$ be the group  of all  diffeomorphisms  on $\T^d$     preserving  the   orientation and    volume  and isotopic to the identity, i.e.,~$\textup{SDiff}(\T^d)$ is the set of all functions $\psi:\T^d\to\T^d$ such that     there  is a path  $I\in W^{1,\ty}(J_1,  C^1(\T^d))$ with  $I(0,x)=x$, $I(1,x)=\psi(x)$ for all $x\in \T^d$, and $I(t,\cdot )$ is a diffeomorphism on $\T^d$ preserving  the   orientation and    volume for all~$t\in J_1$.

 \section{Preliminaries}

  \subsection{Particle
trajectories}\label{S:2.1} 
In this section,  we study  some existence and stability properties for        the Lagrangian trajectories.       
Let us fix a time $T>0$ and an integer  $d\ge1$.
 For  any    vector field $u \in L^1(J_T, C^{1}(\T^d,\R^d) ) $, we  consider     the following ordinary differential equation  in~$\T^d$
 \begin{equation}\label{1.2.1}
\dot x=u(t,x).  
\end{equation} By standard methods, one can show    that for any $y\in \T^d$ this equation   admits a unique solution $x\in W^{1,1}(J_T,\T^d)$ such that $x(0)=y$ (e.g., see  Chapter~1 in~\cite{HJ-1988} and Section~2.1 in~\cite{shirikyan-2008b}). Moreover,  if   $\phi^u_t: \T^d\to \T^d, t\in J_T$ is the corresponding flow  sending $  y $ to $x(t)  $,      then $\phi^u_t$ is a $C^1$-diffeomorphism on $\T^n$ and 
\begin{equation}\label{E:contin}
      \textup{   $ \phi:  L^1(J_T, C^{1}(\T^d,\R^d) )\to C(J_T,C^{1}(\T^d)), \quad   u\mapsto \phi^{u}_\cdot \quad $  is continuous.}
      \end{equation}
 We shall also use the following stability property with respect to a weaker norm (cf. Chapter 4 in \cite{MR0686793}). 
\begin{lemma}\label{L:1.1}     For any  $\la\in(0,1]$ and  $R>0$, there is  $C:=C(R, \la, T)>0$ such that  
 \begin{equation}\label{1.2.sssq}
\| \phi^u-\phi^{\hat u}\|_{L^\ty(J_T,C^{1}(\T^d))} \le C \lvvvert u-\hat u \rvvvert_{T,C^{1}(\T^d,\R^d)}^{\la/2}  
\end{equation} for any $u, \hat u   \in {L^\ty(J_T, C^{1,\la }(\T^d,\R^d) ) }$ verifying  
\begin{gather}
 \lvvvert u-\hat u \rvvvert_{T,C^{1}(\T^d,\R^d)}<1,\label{arajinanhavze}\\
\|u\|_{{L^\ty(J_T, C^{1,\la }(\T^d,\R^d) ) }}+ \|\hat u\|_{{L^\ty(J_T, C^{1,\la }(\T^d,\R^d) ) }} \le R. \label{arajinanhavze2}
\end{gather}
\end{lemma}
\begin{proof} We shall regard $u$ and $\hat u$ as functions on $\R^d$ which are  $2\pi$-periodic in each variable.     
Clearly, it suffices to prove this lemma in the case when $\T^d$ is replaced by~$\R^d$ and $\phi^u_t, \phi^{\hat u}_t: \R^d\to \R^d, t\in J_T$ are the   flows corresponding to $u$ and $\hat u$.

\smallskip
{\it Step~1}. Let us show that there is a constant $C:=C(R, T)>0$ such that
 \begin{equation}\label{aaazez}
\| \phi^u-\phi^{\hat u}\|_{L^\ty(J_T\times \R^d)} \le  C  \lvvvert u-\hat u \rvvvert_{T,L^\ty(\R^d)}^{1/2}.    \end{equation}  
Indeed, we have \begin{align}\label{E:gronw}
\| \phi_t^u-\phi_t^{\hat u}\|_{L^\ty(\R^d)} &= \left\|  \int_0^t (u(s, \phi_s^u)-\hat u(s, \phi_s^{\hat u})) \dd s \right\|_{L^\ty(\R^d)}\nonumber \\&\le\left\|  \int_0^t (u(s, \phi_s^u)-  u(s, \phi_s^{\hat u})) \dd s \right\|_{{L^\ty(\R^d)} }\nonumber\\&\quad+\left\|  \int_0^t (  u(s, \phi_s^{\hat u})-\hat u(s, \phi_s^{\hat u})) \dd s \right\|_{{L^\ty(\R^d)} }=: G_1+G_2.\end{align} 
Then 
 \begin{equation}\label{qsdfsgf}
G_1\le  \|u\|_{L^\ty(J_T, C^{1} (\R^d))} \int_0^t\! \| \phi_s^u-\phi_s^{\hat u}\|_{{L^\ty(\R^d)} } \dd s.\end{equation}
 To estimate $G_2$, let us first note that 
  $$
  |\phi_{t_1}^{\hat u}(y)-\phi_{t_2}^{\hat u}(y)|=\left|\int_{t_1}^{t_2} \hat u(s, \phi_s^{\hat u}(y)) \dd s\right|\le |t_2-t_1|  \, \|\hat u\|_{L^\ty(J_T\times \R^d)}
  $$ for any   $y\in \R^d$ and  $t_1,t_2\in J_T$. 
  Hence
  for any $\eta>0$,
  \begin{equation}\label{E:eeaefhh}
  \sup_{t_1,t_2\in J_T, |t_1-t_2|\le \eta} \|\phi_{t_1}^{\hat u}-\phi_{t_2}^{\hat u}\|_{L^\ty(\R^d)} \le \eta \|\hat u\|_{L^\ty(J_T\times \R^d)}. 
  \end{equation} Taking a partition
   $\tau_i=it/n, i=0, \ldots,n$, we write  
\begin{align}\label{dsdsffytG221}
  G_2&\le  \sum_{i=1}^n\left\|  \int_{\tau_{i-1}}^{\tau_i} (  u(s, \phi_s^{\hat u})-  u(s, \phi_{\tau_{i-1}}^{\hat u})) \dd s \right\|_{{L^\ty(\R^d)} }
\nonumber\\&\quad +\sum_{i=1}^n\left\|  \int_{\tau_{i-1}}^{\tau_i} ( \hat  u(s, \phi_s^{\hat u})-\hat u(s, \phi_{\tau_{i-1}}^{\hat u})) \dd s \right\|_{{L^\ty(\R^d)} }
\nonumber\\&\quad+\sum_{i=1}^n\left\|  \int_{\tau_{i-1}}^{\tau_i} (    u(s, \phi_{\tau_{i-1}}^{\hat u})-\hat u(s, \phi_{\tau_{i-1}}^{\hat u})) \dd s \right\|_{{L^\ty(\R^d)} }\nonumber\\&=:G_{2,1}+G_{2,2}+G_{2,3}.
   \end{align}To estimate $G_{2,1}+G_{2,2}$, we use \eqref{E:eeaefhh}: 
   \begin{equation}\label{uxxumertz}
   G_{2,1}+G_{2,2}\le \frac{T^2}{n}   \|\hat u\|_{L^\ty(J_T\times \R^d)} (\|u\|_{L^\ty(J_T, C^{1}(\R^d) )}+\|\hat u\|_{L^\ty(J_T, C^{1}(\R^d) )}).
\end{equation}We use the relaxation norm defined by \eqref{rexaxX} to bound $G_{2,3}$:
\begin{align*} 
  \left|  \int_{\tau_{i-1}}^{\tau_i} (    u(s, \phi_{\tau_{i-1}}^{\hat u})-\hat u(s, \phi_{\tau_{i-1}}^{\hat u})) \dd s \right|  &\le    \left|  \int_0^{\tau_i} (    u(s, \phi_{\tau_{i-1}}^{\hat u})-\hat u(s, \phi_{\tau_{i-1}}^{\hat u})) \dd s \right|  \nonumber\\&\quad+  \left|  \int_0^{\tau_{i-1}} (    u(s, \phi_{\tau_{i-1}}^{\hat u})-\hat u(s, \phi_{\tau_{i-1}}^{\hat u})) \dd s \right|
  \nonumber\\& \le2 \lvvvert u-\hat u \rvvvert_{T,L^\ty(\R^d)},
   \end{align*}hence 
$$   G_{2,3}\le  2n \lvvvert u-\hat u \rvvvert_{T,L^\ty(\R^d)}.
$$Combining this with   \eqref{dsdsffytG221} and \eqref{uxxumertz}, we get
\begin{align}\label{GGG222}
  G_2\le &  \frac{T^2}{n}   \|\hat u\|_{L^\ty(J_T\times \R^d)} (\|u\|_{L^\ty(J_T, C^{1}(\R^d) )}+\|\hat u\|_{L^\ty(J_T, C^{1}(\R^d) )})\nonumber\\& +2n \lvvvert u-\hat u \rvvvert_{T,L^\ty(\R^d)}.
   \end{align}
If $\lvvvert u-\hat u \rvvvert_{T,L^\ty(\R^d)}=0$, then \eqref{aaazez}  is trivial. Assume that    $\lvvvert u-\hat u \rvvvert_{T,L^\ty(\R^d)}>0$.
   Choosing\footnote{Here $[a]$ stands for the integer part of $a\in\R$.} $n:=[ \lvvvert u-\hat u \rvvvert_{T,L^\ty(\R^d)}^{-1/2}]$, we derive from \eqref{GGG222} and \eqref{arajinanhavze} that 
$$
  G_2\le  C  \lvvvert u-\hat u \rvvvert_{T,L^\ty(\R^d)}^{1/2}.
$$Combining this with  \eqref{E:gronw} and \eqref{qsdfsgf} and applying  the Gronwall inequality, we obtain~\eqref{aaazez}. 

\smallskip
{\it Step~2}. We turn to the proof of \eqref{1.2.sssq}.  
It is easy to verify that  there is a constant  $C_1:=C_1(R, T)>0$ such that 
\begin{equation}\label{ertzrgv}
\|\phi^{\hat u}\|_{L^\ty(J_T,C^1(\R^d))}+\| \p_t \phi^{\hat u}\|_{L^\ty(J_T,C^1(\R^d))}\le   C_1.  
\end{equation}
For   $ j=1, \ldots, d$, we have 
\begin{align}\label{ddsea}
\| \p_j\phi_t^u-\p_j\phi_t^{\hat u}\|_{L^\ty(\R^d)} &= \left\|  \int_0^t ( \lag \nabla u(s, \phi_s^u), \p_j \phi_s^u \rag -\lag \nabla \hat u(s, \phi_s^{\hat u}) , \p_j \phi_s^{\hat u} \rag ) \dd s \right\|_{L^\ty(\R^d)}\nonumber \\&\le \left\|  \int_0^t ( \lag \nabla u(s, \phi_s^u), \p_j \phi_s^u - \p_j \phi_s^{\hat u} \rag ) \dd s \right\|_{L^\ty(\R^d)}\nonumber \\&\quad + \left\|  \int_0^t ( \lag \nabla u(s, \phi_s^u)-\nabla u(s, \phi_s^{\hat u}) , \p_j \phi_s^{\hat u} \rag ) \dd s \right\|_{L^\ty(\R^d)} \nonumber\\&\quad+\left\|  \int_0^t ( \lag \nabla u(s, \phi_s^{\hat u}) - \nabla \hat u(s, \phi_s^{\hat u}) , \p_j \phi_s^{\hat u} \rag ) \dd s \right\|_{L^\ty(\R^d)} \nonumber\\&=: I_1+I_2+I_3.
\end{align}
From   \eqref{arajinanhavze2} it follows that  
$$
I_1\le R \int_0^t\| \p_j\phi_s^u-\p_j\phi_s^{\hat u}\|_{L^\ty(\R^d)} \dd s.
$$
Using \eqref{arajinanhavze2}, \eqref{ertzrgv}, and \eqref{aaazez},
we get   
\begin{align*}
I_2&\le C_1 R  \int_0^t \|  \phi_s^u- \phi_s^{\hat u}\|_{L^\ty(\R^d)} ^\la \dd s \le  C_2     \lvvvert u-\hat u \rvvvert_{T, L^\ty(\R^d)}^{\la/2}.\end{align*}To estimate $I_3$, we integrate by parts and use \eqref{ertzrgv}
\begin{align*}
I_3&\le \left\|     \lag \int_0^t (\nabla u(s, \phi_s^{\hat u}) - \nabla \hat u(s, \phi_s^{\hat u})) \dd s, \p_j \phi_t^{\hat u} \rag    \right\|_{L^\ty(\R^d)}\\&\quad +\left\|  \int_0^t \left( \lag \int_0^s(\nabla u(\theta, \phi_\theta^{\hat u}) - \nabla \hat u(\theta, \phi_\theta^{\hat u}))\dd \theta ,  \p_j (\p_t \phi_s^{\hat u} )\rag \right) \dd s \right\|_{L^\ty(\R^d)}\\&\le C_3 \sup_{s\in [0,t]} \left\|    \int_0^s(\nabla u(\theta, \phi_\theta^{\hat u}) - \nabla \hat u(\theta, \phi_\theta^{\hat u}))\dd \theta   \right\|_{L^\ty(\R^d)}.
\end{align*}Repeating the arguments    of the  proof of \eqref{GGG222} and using the fact that $\nabla   u$ and~$\nabla \hat u$ are H\"older continuous with   exponent $\la$, we obtain that
\begin{align*}
 \sup_{s\in [0,t]} \left\|    \int_0^s(\nabla u(\theta, \phi_\theta^{\hat u}) - \nabla \hat u(\theta, \phi_\theta^{\hat u}))\dd \theta   \right\|_{L^\ty(\R^d)}&\le \frac{C_4}{n^\la}+2n \lvvvert u-\hat u \rvvvert_{T,C^1(\R^d)}\\&\le C_5 \lvvvert u-\hat u \rvvvert_{T,C^1(\R^d)}^{\la/2}
\end{align*} for $n:=[ \lvvvert u-\hat u \rvvvert_{T,C^1(\R^d)}^{-1/2}]$. Combining this with the estimates for $I_1$, $I_2$ and~\eqref{ddsea},  and applying the Gronwall inequality, we arrive at the required result. 
  \end{proof}
By the Liouville theorem (see Corollary~1 in~\cite[p.198]{Arnold78}),  if we   assume additionally that $u$ is divergence-free,  then the flow $\phi_t^u$ preserves the orientation and the volume.  Thus  if       $u\in L^\ty(J_T, C^{1}_\sigma(\T^d,\R^d) ) $,  then  $\phi^u_t\in \textup{SDiff}(\T^d)$ for any~$t\in J_T$. 
 The following proposition shows that,  using a suitable divergence-free field $u$, the flow $\phi^u_t$ can be  driven approximately   to any position $\psi \in \textup{SDiff}(\T^d)$ at time $T$.   
\begin{proposition}\label{P:1.2} 
For any $\e>0, k> 1+ d/2$, $ u_0,u_1\in  H^{k}_\sigma(\T^d,\R^d)$, and $\psi \in \textup{SDiff}(\T^d),$  there is a vector field   $u\in C^\ty(J_T, H^{k}_\sigma(\T^d,\R^d))$ such that $u(0)=u_0, u(T)=u_1$, and 
$$\|\phi^u_T-\psi\|_{C^1(\T^d)}<\e.$$
\end{proposition}

\begin{proof} 
{\it Step~1}. We first  forget about the endpoint conditions   $u(0)=u_0, u(T)=u_1$  and show  that    there is a divergence-free vector field $\hat u\in C^\ty (\R\times\T^d,\R^d)$ such that $$\|\phi^{\hat u}_T-\psi\|_{C^1(\T^d)}<\e/2.$$
 Since $\psi \in \textup{SDiff}(\T^d)$,  there  is a path  $I\in W^{1,\ty}(J_T,  C^1(\T^d))$ such that   $I(0,x)=x$, $I(T,x)=\psi(x)$ for all $x\in \T^d$, and $I(t,\cdot )$ is a $C^1$-diffeomorphism on $\T^d$ preserving  the   orientation and the  volume for all $t\in J_T$. Let us  define the vector field  $\hat  u(t,x)=\p_t I(t,I^{-1}(t,x))$. Then we have  $\hat u\in L^\ty(J_T, C^1( \T^d,\R^d))$ and $I(t,x)=\phi_t^{\hat u}(x), t\in J_T$. As $\phi_t^{\hat u}$   preserves  the   orientation and the  volume,  for  any $g\in C^1(\T^d,\R)$ 
\begin{align*} 
0&=\frac{\dd}{\dd t}\int_{\T^d} g(\phi^{\hat u}_t(y))\dd y= \int_{\T^d} \lag \nabla g(\phi^{\hat u}_t(y)), \p_t \phi^{\hat u}_t(y)\rag \dd y \nonumber\\&=\int_{\T^d} \lag \nabla g(\phi^{\hat u}_t(y)), {\hat u}(t, \phi^{\hat u}_t(y))\rag \dd y=\int_{\T^d}     g(y) \diver  {\hat u}(t, y)     \dd y.
\end{align*} 
 This shows that $ \hat u $ is divergence-free. Taking a    sequence
      of mollifying kernels    $\rho_n\in C^\ty(\R\times\T^d,\R), n\ge1$,    we 
   consider      $\hat u_n:= \rho_n*\hat u=(\rho_n*\hat u_1,\ldots, \rho_n*\hat u_d)   \in C^\ty (\R\times\T^d,\R^d)$. Then $\hat u_n$ is     also divergence-free, since $\diver \hat u_n=\rho_n*\diver\hat u=0$, and $\|\hat u_n-\hat u\|_{L^\ty(J_T,C^1(\T^d,\R^d))}\to 0$ as $n\to\ty$.  By \eqref{E:contin}, this  implies that  $\|\phi_T^{\hat u_n}-\phi_T^{\hat u}\|_{C^1(\T^d)}\to 0$ as $n\to\ty$.    Since  $\phi_T^{\hat u}=\psi$, we get the required result with $\hat u=\hat u_n$ for sufficiently large~$n\ge1$. 

\smallskip
{\it Step~2}. By the   Sobolev embedding,~$H^k\subset C^{1}(\T^d)$ for $ k>1+d/2$ (e.g., see~\cite{adams1975}).
For any~$\de>0$, we take  an arbitrary   $  u\in $ $ C^\ty(J_T, H^{k}_\sigma(\T^d,\R^d))$ satisfying 
\begin{align*}
&u(0)=u_0, \quad   u(T)=u_1,\\
 &\|u-\hat u\|_{L^1(J_T, C^1(\T^d, \R^d))} < \de. 
\end{align*}Then by Step 1 and \eqref{E:contin}, we have
\begin{align*}
\|\phi^u_T-\psi\|_{C^1(\T^d)}&\le \|\phi^u_T-\phi^{\hat u}_T\|_{C^1(\T^d)}+\|\phi^{\hat u}_T-\psi\|_{C^1(\T^d)}\\&< \e/2+\e/2=\e
\end{align*}
  for sufficiently small $\de>0$.

 \end{proof}

\subsection{Existence of strong solutions 
 }\label{S:lav}
 In what follows, we shall assume that $d=3, k\ge3$, and  $\nu = 1$.  In this section, we prove a perturbative result on   existence of strong solutions for
   the
   evolution equation   
\begin{equation}\label{E:21}
\dot u +   Lu + B(u) =g,
\end{equation}where    
  $B(a,b):=\Pi\{\langle a,\nabla\rangle b\}$ and 
$B(a):=B(a,a)$.  Along with \eqref{E:21},    we 
consider the  following more general equation 
\begin{equation}
\dot{u}+L(u+\zeta)+B(u+\zeta)=g.
\label{E1:pert2}\end{equation}Let us fix any    $T>0$   and introduce the space $\XX_{T,k}:= C(J_T,H^{k}_\sigma)\cap L^2(J_T,H_\sigma^{k+1})$ endowed with the norm 
$$
\|u\|_{\XX_{T,k}}:=\|u\|_{L^\ty(J_T,H^{k})}+\|u\|_{ L^2(J_T,H^{k+1})}.
$$ The following   result is a version  of Theorem~1.8 and Remark~1.9 in~\cite{shirikyan-cmp2006} and Theorem~2.1 in \cite{Hayk-2010} in the case of  the 3D NS system in the spaces $H^k, k\ge3$. For the sake of completeness,  we give all the details of the proof,  even though it is very close to the proofs of the previous results.
\begin{theorem}\label{T:pert}
 Suppose that for some functions  $\hat u_0\in H^{k}_\sigma$,   $\hat \zeta \in L^4
(J_T,H^{k+1}_\sigma)$, and $\hat g\in L^2(J_T,H^{k-1}_\sigma)$ problem
(\ref{E1:pert2}), (\ref{ic}) with $u_0=\hat u_0$, $\zeta=\hat\zeta$, and
$g=\hat g$ has a solution $\hat u\in \XX_{T,k}$. Then there are
positive constants $\delta$ and $C$ depending only~on  
\begin{equation}
\|\hat \zeta\|_{L^4(J_T,H^{k+1})}+\|\hat g\|_{L^2(J_T,H^{k-1})}+ \|\hat u\|_{\XX_{T,k}}
   \nonumber
\end{equation}
such that the following statements hold.

\begin{enumerate}
\item[(i)] If $u_0\in H^{k}_\sigma,$ $\zeta\in L^4(J_T,H^{k+1}_\sigma)$, and $g\in L^2(J_T,H^{k-1}_\sigma)$  satisfy the
inequality 
\begin{equation}\label{E1:delt2}
\|u_0-\hat u_0\|_{k}+ \|\zeta-\hat\zeta\|_{L^4(J_T,H^{k+1})} + \|g-\hat g\|_{L^2(J_T,H^{k-1})}<
\delta,
\end{equation}
then problem (\ref{E1:pert2}), (\ref{ic}) has a unique
solution $ u\in \XX_{T,k}.$
\item[(ii)] Let $$\RR:H^{k}_\sigma\times L^4 (J_T,H^{k+1}_\sigma)\times L^2(J_T,H^{k-1}_\sigma)\rightarrow \XX_{T,k}$$
be the operator that takes each triple $(u_0,\zeta,g )$ satisfying
(\ref{E1:delt2}) to the solution $u$ of (\ref{E1:pert2}),
(\ref{ic}). Then
\begin{align}
\|\RR(u_0,\zeta,g )&-\RR(\hat u_0,\hat\zeta,\hat g )\|_{\XX_{T,k}}\le
C\big( \|u_0-\hat u_0\|_{k}\nonumber\\&+
\|\zeta-\hat\zeta\|_{L^4(J_T,H^{k+1})}+\|g-\hat g\|_{L^2(J_T,H^{k-1})}\big).\nonumber
 \end{align}
 \end{enumerate}
\end{theorem}

\begin{proof} 
We use the following standard estimates for
the bilinear form $B$
\begin{align}
\|B(a,b)\|_k\leq C\|a\|_k\|b\|_{k+1}  \qquad &\text{for  }
k\ge 2\label{E1:B1},\\
|(B(a,b),L^k b)|\leq C\|a\|_k\|b\|^2_{k}  \qquad &\text{for  }
k\ge 3\label{E1:B2}
\end{align}
for any $a\in H_\sigma^k$ and $ b \in H_\sigma^{k+1}$ (see Chapter 6 in 
\cite{CF1988}).
We are looking for a solution of  (\ref{E1:pert2}), (\ref{ic}) of the form  $u=\hat u+w$. We have the following equation   for $w$:
\begin{align}
&\dot{w}+L(w+\eta)+B(w+\eta,\hat u+\hat\zeta)+B(\hat u +\hat \zeta,w+\eta)
+B(w+\eta)=q,\nonumber\\
&w(0,x)=w_0(x),\label{E1:eqwin}
\end{align}
where   $w_0:=u_0-\hat u_0$, $\eta:=\zeta-\hat \zeta$, and $q:=g-\hat g$. Setting $\tilde B(u,v):=B(u,v)+B(v,u) $,  we get that 
\begin{equation}
 \dot{w}+Lw+B(w)+\tilde{B}(w,\eta)+\tilde{B}(w,\hat u)+\tilde{B}(w,\hat \zeta)\!=\!q-(L\eta+B(\eta)+\tilde{B}(\hat u,\eta)+\tilde{B}(\hat \zeta,\eta)),\label{E1:pert5}
\end{equation} Using \eqref{E1:B1}, we see that for any $\e>0$, we can choose $\de\in(0,1)$ in  (\ref{E1:delt2}) such that 
$$
\|w_0\|_k+\|q-(L\eta+B(\eta)+\tilde{B}(\hat u,\eta)+\tilde{B}(\hat \zeta,\eta))\|_{L^2(J_T, H^{k-1})}<\e.
$$
 Then, using some standard methods, one gets     that  system (\ref{E1:pert5}), (\ref{E1:eqwin}) has  a solution~$w\in \XX_{T,k}$  for sufficiently small $\e>0$ (see   Section 4 of Chapter~17 in  \cite{taylor1996}).  

\medskip
To prove $(ii)$, we multiply   (\ref{E1:pert5}) by $L^{{k}}w$ and use estimates 
(\ref{E1:B1}) and  (\ref{E1:B2}) 
\begin{align*}
\frac{1}{2}\frac{d}{dt}\|w\|^2_{{k}}+&\|w\|^2_{{k+1}}  \leq
C\bigg(\|w\|^3_{{k}}+ \|w\|_{{k+1}}\|w\|_{{k}}\big(
\|\eta\|_{k}+\|\hat u\|_{k} +\|\hat\zeta\|_{k}\big)\nonumber\\
& +\|w\|_{{k+1}}\big( \|q\|_{k-1}+\|\eta\|_{k+1}+
\|\eta\|_{k}(\|\eta\|_{k-1}+\|\hat u\|_{k}+\|\hat \zeta\|_{k}) \big)\bigg)
. \end{align*}This implies that 
\begin{align*}
\frac{1}{2}\frac{d}{dt}\|w\|^2_{{k}}+\frac{1}{2}&\|w\|^2_{{k+1}}  \leq
C_1\bigg(\|w\|^3_{{k}}+ \|w\|_{{k}}^2\big(
\|\eta\|_{k}^2+\|\hat u\|_{k}^2 +\|\hat\zeta\|_{k}^2\big)\nonumber\\
& +  \Big[ \|q\|_{k-1}^2+\|\eta\|_{k+1}^4 
  +\|\eta\|_{k}^2 (\|\hat u\|_{k}^2+\|\hat \zeta\|_{k}  ^2)\Big]\bigg)
.
\end{align*}
 Integrating this inequality and setting $$A:=\|w_0\|_k^2+\int_0^T\Big[ \|q\|_{k-1}^2+\|\eta\|_{k+1}^4 
  +\|\eta\|_{k}^2 (\|\hat u\|_{k}^2+\|\hat \zeta\|_{k}  ^2)\Big]\dd t,$$  we obtain
  \begin{equation}
  \|w\|^2_{{k}} +\int_0^t \|w\|^2_{{k+1}}    \leq 2A+
2 C_1 \int_0^t\bigg(\|w\|^3_{{k}}+ \|w\|_{{k}}^2\big(
\|\eta\|_{k}^2+\|\hat u\|_{k}^2 +\|\hat\zeta\|_{k}^2\big)\bigg)\dd t
.\label{E1:14}
\end{equation}
  By \eqref{E1:delt2}, we have that $ \|\eta\|_{L^4(J_T,H^{k+1})}\le \de<1$. So the Gronwall inequality gives that
 \begin{equation}
\|w\|_{{k}}^2\leq C_2A+ C_2\int_0^t \|w(s,\cdot)\|_{{k}}^3\dd s
, \quad    t\in J_T, \label{noranhav321}
\end{equation}
where  $C_2>0$ depends only on  
$\|\hat u\|_{L^2(J_T,H^{k})}+\|\hat\zeta\|_{L^2(J_T,H^{k})}$.  Let us denote $$\Phi(t):=A+ \int_0^t \|w(s,\cdot)\|_{{k}}^3\dd s,\quad t\in J_T.$$ Since the case $A=0$ is trivial, we can assume that $A>0$ and $\Phi(t)>0$ for all $t\in J_T$. Thus \eqref{noranhav321} can be written as
$$
 \left(\dot\Phi(t)\right)^{2/3}  \le  C_2 \Phi(t) ,
$$which   is equivalent to
$$
\frac{\dot\Phi(t)}{\left(\Phi(t)\right)^{3/2}} \le C_3, \quad C_3:= C_2^{3/2}.
$$Integrating this inequality,  we get
 \begin{equation}\label{akzlexbd}
\Phi(t)\le \frac{A}{(1-tC_3\sqrt{A}/2)^2}\le 4A \,\,\,\text{for
any }\,t\le\frac{1}{C_3\sqrt{A}}.
\end{equation}  Choosing  $\delta>0$ so small   that $\frac{1}{C_3\sqrt{A}}\geq T$ and using \eqref{noranhav321} and \eqref{akzlexbd}, we obtain 
$$
\|w(t)\|_{{k}}^2\le C_2 \Phi(t) \le   4C_2A \,\,\,\text{for
any }\,t\in J_T.
$$  
    Combining this with  \eqref{E1:14},  we get   for any $t\in J_T$
\begin{equation*}
  \|w\|^2_{{k}} +\int_0^t \|w\|^2_{{k+1}}    \leq C_4A \le C_5  \big( \|w _0 \|_{k}^2 +
\|\eta\|_{L^4(J_T,H^{k+1})}^2+\|q\|_{L^2(J_T,H^{k-1})}^2\big).
\end{equation*}
 This completes the proof  of the theorem.  \end{proof}

% \section{Main results} 
 \section{Approximate controllability of the NS  system}
\label{s2}
 In this section, we state the main results of this paper.  Let us fix any $T>0$ and $k\ge3$, and consider the NS   system
\begin{align}
\dot{u}+Lu+B (u)&=h(t)+\eta(t),\,\,\,
\label{2.1}\\
u(0,x)&=u_0(x),\label{2.2}
\end{align}
where $h\in L^2  (J_T,H^{k-1}_\sigma)$ and $u_0\in
H^k_\sigma$ are given functions  and $\eta$ is a control taking values in a finite-dimensional space $E\subset H^{k+1}_\sigma$.   
 We denote by $\Theta(h,u_0)$ the set of functions~$\eta\in
L^2(J_T,H^{k-1}_\sigma)$ for which (\ref{2.1}), (\ref{2.2}) has
a solution~$u$ in~$\XX_{T,k}$.  
 By Theorem~\ref{T:pert},      $\Theta(h,u_0)$ is an open subset of
$L^2(J_T,H^{k-1}_\sigma)$.  
 Recall that $\RR(\cdot,\cdot,\cdot)$ is the operator defined in Theorem~\ref{T:pert}. To simplify   notation, we write $\RR(u_0,h+\eta)$ instead of
$\RR(u_0,0,h+\eta)$ for any  $\eta\in \Theta(h,u_0)$. The     embedding~$H^3\subset C^{1,1/2}$   implies that the flow $\phi_t^{\RR(u_0,h+\eta)}$ is  well defined for any~$t\in J_T$. 
We set $$Y_{T,k}:= \XX_{T,k}\cap W^{1,2}(J_T, H^{k-1}_\sigma).$$
 We shall use the following notion of   controllability.
\begin{definition}\label{D:1} Equation (\ref{2.1}) is said to be {\it approximately controllable}
at time~$T$ by an $E$-valued control  if for any $\e>0$ and       any $\varphi\in Y_{T,k}$    there is a control $\eta
\in \Theta(h,u_0)\cap L^2(J_T,E) $ such that
\begin{equation}\label{uxxum123}
\|\RR_T(u_0,h+\eta) - \varphi(T) \|_{k}+ \lvvvert  {\RR(u_0,h+\eta)} - \varphi \lvvvert_{ T,  k}+\|\phi^{\RR(u_0,h+\eta) }- \phi^{\varphi} \|_{L^\ty(J_T,C^1)}<\e,
\end{equation}where $u_0=\varphi(0)$ and $\lvvvert \cdot \lvvvert_{ T,  k}:=\lvvvert  \cdot \lvvvert_{ T, H^k}$.   
\end{definition}  
 Let us recall some notation introduced in \cite{AS-2005,AS-2006}, and  \cite{shirikyan-cmp2006}. For any finite-dimensional subspace
$E\subset H_\sigma^{k+1}$, we denote by $\FF(E)$ the largest
vector space $F\subset H^{k+1}_\sigma$ such that for any
$\eta_1\in F$ there are vectors\footnote{ The integer $ p$ may depend  on $\eta_1$.} $\eta, \zeta^1,\ldots ,\zeta^ p\in
E$   satisfying
the relation
\begin{equation}\label{E2:sahmnf}
\eta_1=\eta-\sum_{i=1}^{ p} B(\zeta^i).
\end{equation}
 As $E$ is a
finite-dimensional subspace and $B$ is a bilinear operator, the set of all vectors $\eta_1\in H^{k+1}_\sigma$ of the form \eqref{E2:sahmnf} is contained
  in a finite-dimensional space.  It is easy
to see that if subspaces $G_1,G_2  \subset H_\sigma^{k+1} $ are composed of elements $\eta_1$ of the form 
(\ref{E2:sahmnf}), then so does $G_1 + G_2$.  Thus the   space $\FF(E)$ is well defined.  We define $E_j$ by the
rule
\begin{equation}\label{Esahm}
E_0=E,\quad E_j=\FF(E_{j-1})\quad \textrm{for} \quad j \geq
1,\quad E_\infty=\bigcup_{j=1}^\infty E_j. 
\end{equation} Clearly, $E_j$ is a non-decreasing sequence of subspaces. We   say that $E$ is {\it saturating} in $ H^{k-1}_\sigma$  if $E_\ty$  is dense in $ H^{k-1}_\sigma$.
The following theorem is the main result of this paper.
\begin{theorem}\label{T.2.1}
Assume that $E 
$ is  a finite-dimensional subspace of $H_\sigma^{k+1}$ and $h\in L^2(J_T,H^{k-1}_\sigma)$.  If   $E 
 $ is saturating in  $ H^{k-1}_\sigma$, then~(\ref{2.1})  is approximately controllable  at   time $T$ by controls  
$\eta\in C^\infty(J_T,E)$ in the sense of Definition~\ref{D:1}.
\end{theorem}
We have the following two corollaries of this result. 

\begin{corollary}\label{D.3.1} Under the conditions of Theorem \ref{T.2.1},   if    $E 
 $ is saturating in~$ H^{k-1}_\sigma$, then       for any $\e>0$,  $u_0,  u_1 \in
H^{k}_\sigma$, and   $\psi\in \textup{SDiff}(\T^3)$   there is a control $\eta
\in \Theta(h,u_0)\cap C^\ty(J_T,E) $ such that
$$
\|\RR_T(u_0,h+\eta) - u_1 \|_{k}+ \| \phi_T^{\RR(u_0,h+\eta)} - \psi\|_{  C^1}<\e.
$$
\end{corollary}  
Let us denote by $\textup{VPM}(\T^3)$ the set of all volume-preserving mappings from~$\T^3$ to $\T^3$. 
According\footnote{ The result of \cite{MR1956851} is stated for a cube, but it remains valid also in the case of a torus.  } to Corollary 1.1 in \cite{MR1956851}, we have that $\textup{VPM}(\T^3)$ is the closure of $\textup{SDiff}(\T^3)$  in $L^p(\T^3)$ for any $p\in [1,+\ty)$. Thus we get the following result.
\begin{corollary}Under the conditions of Theorem \ref{T.2.1},   if    $E 
 $ is saturating in~$ H^{k-1}_\sigma$, then       for any $\e>0$,    $p\in [1,+\ty)$,     $u_0,  u_1 \in
H^{k}_\sigma$, and     $\psi\in \textup{VPM}(\T^3)$  there is a control $\eta
\in \Theta(h,u_0)\cap C^\ty(J_T,E) $ such that 
$$
\|\RR_T(u_0,h+\eta) - u_1 \|_{k}+ \| \phi_T^{\RR(u_0,h+\eta)} - \psi\|_{  L^p}<\e.
$$

\end{corollary}

The rest of this subsection is devoted to the proofs of Theorem \ref{T.2.1} and Corollary \ref{D.3.1}.
They  are   based on the following result which is proved in Section~\ref{S:reduct}.
\begin{theorem}\label{T.reduct}  Assume   that~$E$ is  an arbitrary finite-dimensional subspace of~$H_\sigma^{k+1}$ and    $h\in L^2(J_T,H^{k-1}_\sigma)$.  Then for any $\e>0$, $u_0\in H^k_\sigma$, and $\eta_1\in \Theta(h,u_0)\cap L^2(J_T, E_1)$
    there is   $\eta \in \Theta(h,u_0)\cap C^\infty(J_T, E )$ such that
  \begin{align*}
\|\RR_T(u_0,h+\eta_1) -& \RR_T(u_0,h+\eta) \|_{k}+  \lvvvert  {\RR(u_0,h+\eta_1)} -  {\RR(u_0,h+\eta)}\lvvvert_{T, k }\\ +&\|\phi^{\RR(u_0,h+\eta_1) }- \phi^{\RR(u_0,h+\eta) } \|_{L^\ty(J_T,C^1)}<\e.     
\end{align*}

\end{theorem}

\begin{proof}  [Proof of Theorem \ref{T.2.1}]
Let us take any $\e>0, \de>0$, and $\varphi\in Y_{T,k}$.   Then $$\eta_0:=\dot \varphi +L\varphi+B(\varphi)-h$$ belongs to $\Theta(u_0, h)$ and $\varphi(t)=\RR_t(u_0,h+\eta_0)$ for any $t\in J_T$, where $u_0=\varphi(0)$. Since $E_\infty$ is dense in $H^{k-1}_\sigma$, we have that
$$ \|P_{E_N} \eta_{ 0} - \eta_{ 0}\|_{L^2(J_T,H^{k-1})} \rightarrow 0 \,\, \text{as}\,\,N \rightarrow \infty,$$
where $P_{E_N}$ is the orthogonal projection onto $E_N$ in $H$.
By Theorem \ref{T:pert}, for  sufficiently large    $N$, we have   $P_{E_N} \eta_{ 0} \in \Theta(h,u_0)$ and 
$$
 \|  {\RR(u_0,h+P_{E_N} \eta_{ _0} )} - \varphi \|_{ X_{T,k}}<\de.
$$
By  \eqref{E:contin}, we can choose    
   $\de>0$ so small that   
$$
\|\phi^{\RR(u_0,h+ P_{E_N}\eta_0 ) }- \phi^{\varphi} \|_{L^\ty(J_T,C^1)}<\e.
$$
  Applying $N$
times Theorem \ref{T.reduct}, we complete the proof of Theorem
\ref{T.2.1}.

 \end{proof}

\begin{proof}  [Proof of Corollary \ref{D.3.1}]
Let us take any $\e>0 $,  $\psi \in \textup{SDiff}(\T^3)$, and  $ u_0,u_1\in  H^{k}_\sigma$. By Proposition \ref{P:1.2},    there is a vector field   $u\in C^\ty(J_T, H^{k}_\sigma )$ such that $u(0)=u_{0}, u(T)=u_{1}$, and 
\begin{equation}\label{eezza}
 \|\phi^{u}_T-\psi\|_{C^1}<\e.
 \end{equation} Applying Theorem \ref{T.2.1},   we find a control $\eta
\in \Theta(h,u_0)\cap C^\ty(J_T,E) $ such that~\eqref{uxxum123} holds  with $\varphi=u$. In particular, 
 
$$
\|\RR_T(u_0,h+\eta) - u(T) \|_{k}+ \| \phi^{\RR(u_0,h+\eta)}_T-\phi^{u}_T\|_{C^1}<\e.
$$ Combining this with \eqref{eezza}, we get the required result.    
  
 \end{proof}

\section{Proof Theorem \ref{T.reduct}} \label{S:reduct}

The proof follows the arguments  of 
 \cite{AS-2005, AS-2006}, and \cite{shirikyan-cmp2006}. We   consider  the following system
\begin{align}
\dot{u}+ L (u+\zeta)+B(u+\zeta)&=h+\eta \label{E2.conv}
\end{align}
  with two $E$-valued controls    $\eta, \zeta$.
 We denote by  $\hat{\Theta}(u_0,h)$ the set of  
$(\eta, \zeta)\in L^2(J_T,H_\sigma^{k-1})\times
L^4(J_T,H_\sigma^{k+1})$ for which problem (\ref{E2.conv}),
(\ref{ic}) has a   solution in~$\XX_{T,k}$.
    Theorem \ref{T.reduct}  is deduced  from the following   proposition which is proved at the end of this section (cf. Proposition 3.2 in \cite{shirikyan-cmp2006}).
     \begin{proposition}\label{P.2}
    For any  $\eta_1 \in   \Theta(u_0,h) \cap L^2(J_T,E_1)$, there is a sequence $(\eta_n,\zeta_ n) \in \hat  \Theta(u_0,h) \cap  C^\infty(J_T,E\times
E)$ such that 
\begin{align}\label{E:ppp}
&\| \RR(u_0,0,h+\eta_1)- \RR(u_0, \zeta_n,h+\eta_n)\|_{ L^\ty(J_T, H^k)}+  \lvvvert  \zeta_n \rvvvert_{T,k} \! \to \!0    \,\,\text{as $n\!\to\!\ty$,}\\ &
\sup_{n\ge1}(\|  \RR(u_0,\zeta_n, h+\eta_n )\|_{\XX_{T,k}} \!+\!\| \zeta_n\|_{L^\ty(J_T ,H^{k+1})}+\| \eta_n\|_{L^2(J_T ,H^{k-1})}) < \infty.\label{sahman}
 \end{align}
\end{proposition}
\begin{proof}[Proof of Theorem \ref{T.reduct}] Let us take any      $u_0\in H^k_\sigma$ and $\eta_1\in \Theta(h,u_0)\cap L^2(J_T, E_1)$, and let     $(\eta_n,\zeta_ n) \in \hat  \Theta(u_0,h) \cap  C^\infty(J_T,E\times
E)$ be  any sequence satisfying \eqref{E:ppp} and  \eqref{sahman}. 
 Let    $\hat \zeta_n \in C^\infty(J_T,E)$  be   such that $\hat\zeta_n(0)=\hat\zeta_n(T)=0$ and   
\begin{align}
&\|\zeta_n-\hat\zeta_n\|_{L^4(J_T,H^{k+1})} \rightarrow 0 \,\, \text {as} \,\,
n\rightarrow \infty,\label{E:deficz}\\ & \sup_{n\ge1}\| \hat\zeta_n\|_{L^\ty(J_T ,H^{k+1})} < +\infty.\label{ddfa}
\end{align}  Then \eqref{E:ppp} and \eqref{E:deficz} imply that
\begin{align}\label{etaqdrf}
   \lvvvert \hat \zeta_ n \lvvvert_{ T,  k} &\le   \lvvvert \hat \zeta_ n-\zeta_ n \lvvvert_{ T,  k}+  \lvvvert   \zeta_ n \lvvvert_{ T,  k}  \nonumber \\&\le \int_0^T \|\hat \zeta_ n(s)-\zeta_ n(s)\|_k\dd s+  \lvvvert   \zeta_ n \lvvvert_{ T,  k}   \to 0  \quad \textup{as $n\to \ty$.} 
  \end{align}
By  Theorem \ref{T:pert} and \eqref{sahman},    for sufficiently large $n\ge1$, we have  $(\eta_n, \hat \zeta_ n) \in \hat  \Theta(u_0,h)$  and
  \begin{align}\label{Eqq2}
\| \RR(u_0, \zeta_n,h+\eta_n)-\RR(u_0, \hat \zeta_n,h+  \eta_n)\|_{\XX_{T,k}}  \to 0  \quad\text {as $n\to\ty$.
}   \end{align} 
Note that    
\begin{align}\label{Eqq221}
 \RR_t(u_0, \hat \zeta_n,h+  \eta_n)&=\RR_t(u_0,   0  , h+  \hat \eta_n)-\hat \zeta_n(t)  \quad \text{for $t\in J_T,$}\\
 \RR_T(u_0, \hat \zeta_n,h+  \eta_n)&=\RR_T(u_0,   0  , h+  \hat \eta_n),\label{Eqq221yeur}
\end{align} where
 $\hat \eta_n:= \eta+\p_t 
{\hat \zeta}_n$.   From \eqref{E:ppp}, \eqref{Eqq2}, and \eqref{Eqq221yeur} it follows  that 
$$
\| \RR_T(u_0,0, h+\eta_1)-\RR_T(u_0,0, h+ \hat \eta_n)\|_{k}  \to 0  \quad\text {as $n\to\ty$.
}   $$  
Using  \eqref{E:ppp}, \eqref{etaqdrf}-\eqref{Eqq221},  we obtain   
  \begin{align*}    \lvvvert   {\RR(u_0, 0, h+\eta_1)} -  \RR(u_0, &  0  ,  h+\hat\eta_n)\lvvvert_{ T,  k}\\&\le  T \|    {\RR(u_0,0, h+\eta_1)} -  {\RR(u_0, \zeta_n, h+ \eta_n)}\|_{  L^\ty(J_T, H^k)}\\& \quad +T\| \RR(u_0, \zeta_n, h+ \eta_n)-\RR(u_0, \hat\zeta_n, h+ \eta_n)\|_{\XX_{T,k}}\\&\quad\, +\lvvvert   {\RR(u_0, \hat \zeta_n, h+ \eta_n)} -  {\RR(u_0,0, h+\hat\eta_n)}\lvvvert_{ T,  k}\\&
 \to 0     \quad\text{as $n\to \ty$}.
\end{align*} Combining this with the embedding $H^3\subset C^{1,1/2}$, \eqref{sahman}, \eqref{Eqq221}, and   applying Lemma \ref{L:1.1} with $\la=1/2$, we get that 
$$   \|\phi^{\RR(u_0, 0,h+\eta_1) }- \phi^{\RR(u_0,0, h+\hat \eta_n) } \|_{L^\ty(J_T,C^1)}\to 0     \quad\text{as $n\to \ty$}.
$$
This completes the proof of Theorem \ref{T.reduct}. 
 \end{proof}

 \begin{proof}[Proof of Proposition \ref{P.2}] {\it Step~1}.  Without loss of generality, we can assume that $\eta_1\in \Theta(u_0,h)\cap  E_1$ is constant. Indeed, the general case  is then obtained by       approximation in $L^2(J_T,H^{k-1}_\sigma)$  by   piecewise constant controls with finite number of intervals of constancy and   successive applications of    the result on the intervals of constancy.

\smallskip
 By the definition of $\FF(E)$, for any $\eta_1\in E_1$, there
are  vectors $\xi^1,\ldots ,\xi^p, \eta \in E$   such that $$
\eta_1=\eta-\sum_{i=1}^p B(\xi^i).$$ Choosing    $m = 2p$ and
$$
 \zeta^i := -\zeta^{i+p} := \frac{1}{\sqrt{2}}\xi^i, \quad i = 1,
\ldots ,p,
$$ it is easy to see that  
 \begin{equation} \label{isver}
B(u)-\eta_1=\frac1m \sum_{j=1}^m\left(  B(u+\zeta^j) + L \zeta^j\right)-\eta\quad
\text{for any }u \in H^1_\sigma. 
\end{equation}
 Then  $u_1:=\RR(u_0,0,h+\eta_1)\in \XX_{T,k}$   satisfies the following  equation
\begin{equation}\label{E3.aftleem1}
\dot{u}_1+Lu_1+\frac1m\sum_{j=1}^m    \left( B(u+\zeta^j)+ L\zeta^j\right)=h(t)+ \eta.
\end{equation}
   Let us define 
$\zeta_n(t)=\zeta(\frac {nt}{T})$, where
$\zeta(t)$  is a $1$-periodic function such that
$$\zeta(s)=\zeta^j\text{ for }  s\in \left[ {(j-1)}/{m},  j/m\right), j=1, \ldots, m.$$ 
 Equation  (\ref{E3.aftleem1}) is
equivalent to  
\begin{equation}
\dot u_1  +L(u_1+\zeta_n) + B(u_1 + \zeta_n)  = h(t) + \eta 
+ f_n(t),\nonumber
\end{equation}
where
\begin{align}\label{E3.fnitesq}
f_n(t):=& L\zeta_n +B(u_1 + \zeta_n) -\frac1m\sum_{j=1}^m     \left( B(u_1+\zeta^j)+ L\zeta^j\right)  .
\end{align}
For any $f\in L^2(J_T, H)$, let us set 
$$
  {K} f(t)=\int_0^t e^{-(t-s)L}f(s)\dd s.
$$ It is easy to check that 
\begin{equation}\label{erttttr}
\textup{$K$ is continuous from $ L^2(J_T, H^{p-1}_\sigma)$ to  $ \XX_{T,p}$  for any $p\ge1$,}
\end{equation}  and $v_n = u_1 - K f_n $ is a solution of the problem
\begin{align}\label{eq_vn}
\dot v_n +L (v_n +\zeta_n)+B (v_n + \zeta_n + K f_n)  & = h(t)
+ \eta , \\
v_n&=u_0.\nonumber
 \end{align}
 
\smallskip
{\it Step~2}. Let us show that 
\begin{equation}\label{E3.fk}
   \|K f_n \|_{ L^\ty(J_T,H^k)}\rightarrow 0 \quad \text{as $n\to\ty$}.
\end{equation}
Indeed, the definition of $\zeta_n$ gives  that 
\begin{equation}
\sup_{n\ge1} \| \zeta_n\|_{L^\ty(J_T ,H^{k+1})}  < \infty.\label{dkjfkcsf}
\end{equation}
 Combining this with \eqref{E3.fnitesq}, \eqref{E1:B1},  and the fact that   $u_1 \in \XX_{T,k}$, we get   
\begin{equation}\label{mmpae}
\sup_{n\ge1}\| f_n\|_{L^\ty(J_T ,H^{k-1})} < \infty.
\end{equation}
This implies  that 
\begin{align}\label{eeza} 
\sup_{n\ge1}  \| K f_n\|_{L^\ty(J_T,H^{k+1/2}_\sigma)} & \le  C\sup_{n\ge1, t\in [0,T]} \int_0^t \|L^{3/4} e^{-(t-s)L}\|_{\LL(H)} \|f_n(s)\|_{k-1}\dd s\nonumber\\&\le C _1\sup_{n\ge1, t\in [0,T]} \int_0^t (t-s)^{-3/4} \|f_n(s)\|_{k-1}\dd s\nonumber\\&\le C_2\sup_{n\ge1}  \|f_n\|_{L^\ty(J_T, H^{k-1})}  <\ty,
\end{align}where we used the inequality\footnote{This inequality is proved with the help of a decomposition in the
eigenbasis  $\{e_j\}$ of $L$:
$$
\|L^{r} e^{-t L}u\|^2=\sum_{j=1}^\ty\alpha_j^{2r}e^{-2t\alpha_j} u_j^2 \le C_{r} t^{-r} \|u\|^2,
$$where $u_j:=\lag u,e_j\rag$ and $\alpha_j$ is the eigenvalue corresponding to $e_j$.}
$$
\|L^{r} e^{-t L}\|_{\LL(H)}\le C_{r} t^{-r} \quad\text{for any $r\ge0, t>0$}.
$$  In Step 4 of the proof of Proposition 3.2 in \cite{shirikyan-cmp2006}, it is established  that 
$$
   \|K f_n \|_{L^\ty(J_T,H^{1})}\rightarrow 0.
$$Using this with  \eqref{eeza} and an interpolation inequality, we get \eqref{E3.fk}. Combining~\eqref{erttttr} with \eqref{mmpae}, we obtain also that
\begin{equation}\label{eertfsqa}
  \sup_{n\ge1} \|K f_n \|_{\XX_{T,k} }< \ty.
\end{equation}
  
\smallskip
{\it Step~3}. Equation \eqref{eq_vn} can be rewritten as
\begin{align}\label{eq_vnaas}
\dot v_n +L (v_n +\zeta_n)+B (v_n + \zeta_n)  & = h(t)
+ \eta +g_n(t),
 \end{align}
 where
 $$
 g_n(t):=-(B (v_n + \zeta_n,  K f_n) +B (  K f_n, v_n + \zeta_n) +B (  K f_n) ).
 $$From \eqref{E3.fk}, \eqref{sahman}, and  \eqref{E1:B1} it is easy to deduce  that $\|g_n\|_{L^2(J_T, H^{k-1})}\to 0$ as~$n\to\ty$.  From \eqref{eertfsqa} it follows that 
$$
  \sup_{n\ge1} \|v_n \|_{\XX_{T,k} }<\ty.
$$
  Therefore,   by
 Theorem \ref{T:pert} and \eqref{dkjfkcsf},   we have $(\eta, \zeta_n)\in \hat  \Theta(u_0,h)$ for sufficiently large $n\ge1$ and     
$$
  \|\RR(u_0,\zeta_n,\eta )-v_n \|_{\XX_{T,k}}\rightarrow 0 \,\, \text{as} \,\,
  n\rightarrow \infty.
$$   
On the other hand, by \eqref{E3.fk},       
$$
  \|v_n-u_1\|_{L^\ty(J_T,H^k)  }\rightarrow 0 \,\, \text{as} \,\,
  n\rightarrow \infty,
$$whence 
\begin{align*}
&  \|\RR(u_0,\zeta_n,\eta )-u_1 \|_{L^\ty(J_T,H^k)}\rightarrow 0 \,\, \text{as} \,\,
  n\rightarrow \infty,\\
   & \sup_{n\ge1} \|\RR(u_0,\zeta_n,\eta ) \|_{\XX_{T,k} }<+\ty.
\end{align*}

\smallskip
{\it Step~4}. 
Let us   show that 
\begin{equation}\label{suja}
\lvvvert \zeta_ n \lvvvert_{ T,  k} \to 0    \quad\text{ as $n\to\ty$.}
\end{equation}We  set $\LL \zeta_n (t) :=\int_0^t  \zeta_ n(s) \dd s$. It suffices to check that  
\begin{enumerate}
\item[(i)] the sequence $\LL  \zeta_ n$ is relatively compact in $C(J_T, H^k_\sigma)$.
\item[(ii)] for any $t\in J_T$, $\LL\zeta_n(t)\to 0$ in $H^k_\sigma$ as $n\to\ty$.  
\end{enumerate} 
To prove the first assertion,  we use the Arzel\`{a}--Ascoli theorem. 
 The functions~$\zeta_n$ are piecewise constant and  the set $\zeta_n(t), t \in J_T$ is contained in a finite
subset of~$H_\sigma^{k+1}$ not depending on $n$. This implies that  there is a
compact set $F \subset H_\sigma^{k+1}$ such that
$$ \LL \zeta_n(t) \in F \text{ for all }t \in J_T ,  n \ge 1.$$
From \eqref{sahman} it follows that 
the sequence $ \LL \zeta_n $ is uniformly equicontinuous on $J_T$.
Thus, by the Arzel\`{a}--Ascoli theorem, $ \LL \zeta_n $ is
relatively compact in $C(J_T, H^k_\sigma)$. 

\smallskip
Let us prove (ii).    Let $t=t_l+\tau$, where
$t_l=\frac{lT}{n}$, $l\in \N$ and $\tau\in [0,\frac{T}{n})$. In view of the construction of 
  $\zeta_n$, we have that $\LL \zeta _n(lT/n)=0$. Combining this with \eqref{sahman}, we get   
$$
\LL \zeta _n(t)= \int_{\frac{lT}{n}}^{t}\zeta _n(s)\dd s\to 0,
$$which completes the proof of \eqref{suja}.

\smallskip

Finally, taking  an arbitrary sequence $\hat \zeta_n\in C^\ty (J_T, E)$ such that $$\|\zeta_n-\hat \zeta_n\|_{L^\ty(J_T,E)}\to 0 \text{ as $n\to \ty$},$$ and using Theorem \ref{T:pert}, we see that the conclusions of Proposition \ref{P.2} hold  for the sequence $(\eta, \hat\zeta_n)\in C^\ty(J_T,  E\times E)$.

  \end{proof}

\section{Examples of saturating spaces}

In this section, we provide three   types of    examples of saturating spaces  which ensure  the controllability of the 3D  NS system in the sense of Definition \ref{D:1}.   

\subsubsection{Saturating spaces associated with   the  generators of $\Z^3$}\label{type1}

 Let us first introduce some notation.   Denote by $ \Z^3_*$ the set of non-zero integer vectors $\ell=(l_1,l_2,l_3)\in\Z^3$.
 For any~$\ell\in \Z^3_*$, let us define the functions
\begin{equation}\label{E.cmsmer}
c_\ell(x) = l(\ell) \cos\langle \ell, x \rangle, \,s_\ell(x) = l(\ell)
\sin\langle \ell, x \rangle,
\end{equation}
where    $\{l(\ell),l(-\ell)\}$  is an arbitrary orthonormal basis in 
$$\ell^\bot:=\{x\in \R^3: \langle x, \ell \rangle=0\}.$$ Then $c_\ell$ and $s_\ell$ are eigenfunctions of $L$ and  the family $\{c_\ell, s_\ell\}_{\ell\in \Z^3_*}$ is an orthonormal basis in $H$. Let $c_0=s_0=0$. For any  subset $\KK\subset \Z^3$, we denote    
\begin{equation}\label{E:EEE}
E(\KK):= \textup{span} \{ c_\ell, c_{-\ell}, s_\ell, s_{-\ell}: \ell\in \KK\}.
\end{equation}When $\KK$ is finite,  the spaces  $E_j(\KK)$ and $E_\ty(\KK)$  are defined  by \eqref{Esahm} with   $E=E(\KK)$. We denote by $\Z^3_\KK$ the set of all vectors $a\in \Z^3$ which can be represented as  finite linear combination of elements of $\KK$ with integer coefficients. 
We shall say that   $\KK\subset\Z^3$ is a {\it generator}  if  $\Z^3_\KK= \Z^3$.
\smallskip
The following theorem provides  a characterisation of saturating spaces of the form  \eqref{E:EEE}.  
\begin{theorem}\label{crit}
For any  finite set  $\KK\subset \Z^3$, we have the equality  
 \begin{equation}\label{E:generato}
E(\Z^3_\KK )= E_\ty(\KK).
\end{equation}
 Moreover, $E(\KK)$ is saturating in $H$ if and only if $\KK$ is a generator of $\Z^3$. If~$E(\KK)$ is  saturating  in $H$,  then it is saturating in  $H^k_\sigma$ for any $k\ge0$.  
\end{theorem}
In  \cite{MR2032128}  a similar result is conjectured in the case of finite-dimensional approximations of the 3D NS system and a proof is given for the saturating property of $E(\KK)$ when    $\KK=\{(1,0,0), (0,1,0), (0,0,1)\}$\footnote{In that case  one has   $\dim E(\KK)=12$. In Proposition \ref{lsdfavt},  we give an example of a 6-dimensional saturating space.}.  A 2D  version of  Theorem~\ref{crit} is established in     \cite{EM01} and   \cite{HM-2006}. In that case,   the set $\KK$ is   a generator of~$\Z^2$  containing  at least two vectors with different Euclidian norms (the reader is referred to the original papers for the exact statement). The proof in the 3D  case, as well as the statement of the result,  differ essentially from the 2D case.

\smallskip
    In view of   Theorem~\ref{crit}, the following simple criterion is useful  for constructing saturating spaces     (see Section 3.7 in \cite{MR780184}).   
\begin{theorem} \label{gij} A set  $\KK\subset \Z^3$ is a generator if and
only if the greatest common divisor of the set $\{\det(a,b,c): a,b,c\in \KK\}$ is $1$, where $\det(a,b,c)$ is the determinant of the 
matrix with rows $a,b$ and $c$.  
\end{theorem} 
 
 The proof of Theorem \ref{crit} is deduced from the following   auxiliary result. 
   \begin{proposition}\label{ppo} Assume that   $\WW\subset \Z^3$  is   a finite set containing  a linearly independent family $\{p,q,r\}\subset \Z^3$.  Then   for any  non-parallel vectors $m,n\in \WW$ we have    $\aA_{m\pm n}, \BB_{m\pm n} \subset  E_3(\WW)$, where 
 $$
 \aA_\ell:=\textup{span}\{c_\ell,c_{-\ell}\}, \quad\BB_\ell:=\textup{span}\{s_\ell,s_{-\ell}\}, \quad \ell\in \Z^3_*.
 $$
 \end{proposition} %Taking this proposition    for granted, let us establish Theorem \ref{crit}.
 \begin{proof}[Proof of Proposition \ref{ppo}]   We shall confine ourselves to the proof of the inclusion
  \begin{equation}\label{U:a10}
 \aA_{m+n}\subset E_3(\WW).
 \end{equation} The other conclusions in  the proposition are  checked  in the same way. 

\smallskip

{\it Step~1}. We shall write   $m\nparallel n$ when the vectors $m,n \in \R^3$ are non-parallel.  For any   $m,n\in \WW$ such that  $m\nparallel n$, let us denote by $ \de:=\de(m,n)$ one of the two unit vectors belonging to $m^\bot \cap n^\bot$.   In this step we   show that 
        \begin{equation}\label{uxxahay}
    \de \cos \lag m\pm n,x\rag  , \de \sin \lag m\pm n,x\rag\in  E_1(\WW).
    \end{equation}
 Indeed, for any $a\in \R^3_*$, let us denote by $P_a$ the orthogonal projection in $\R^3$ onto   $a^\bot$. Then we have 
$$
\Pi (a \cos\lag l,x \rag)= (P_l a) \cos\lag l,x \rag, \quad \Pi (a \sin\lag l,x \rag)= (P_l a) \sin\lag l,x \rag
$$for any   $l\in \Z^3_*$.
Combining this with some trigonometric identities and the definition of $B$, one gets that 
\begin{align} 
2B(a \cos \lag m,x\rag+ b\sin \lag n,x\rag)&= \cos\lag m-n,x \rag P_{m-n} \left(\lag a,n\rag b- \lag b,m\rag a \right)\nonumber\\&\quad +\cos\lag m+n,x \rag P_{m+n} \left(\lag a,n\rag b+ \lag b,m\rag a \right),\label{anh1}
\end{align}for any $a\in m^\bot$ and $b\in n^\bot$ 
(see Step 1 of the proof of Proposition 2.8 in~\cite{shirikyan-cmp2006}). This implies that
\begin{align} 
2B(b \cos \lag n,x\rag+ a\sin \lag m,x\rag)&=- \cos\lag m-n,x \rag P_{m-n} \left(\lag a,n\rag b- \lag b,m\rag a \right)\nonumber\\&\quad +\cos\lag m+n,x \rag P_{m+n} \left(\lag a,n\rag b+ \lag b,m\rag a \right).\label{anhz1}
\end{align}
Taking the sum of \eqref{anh1} and \eqref{anhz1}, we obtain that 
\begin{align}\label{cru1} \cos\lag m+n,x \rag P_{m+n} \left(\lag a,n\rag b+ \lag b,m\rag a \right)&= 
 B(a \cos \lag m,x\rag+ b\sin \lag n,x\rag)\nonumber  \\&\quad +
 B(b \cos \lag n,x\rag+ a\sin \lag m,x\rag).
\end{align} 
Let us fix any   $\la\in \R$ and choose in this equality $a=\de$ and~$\lag b, m\rag=\la$. This choice is possible since $m\nparallel n$. Then we have 
$$
\la \de\cos\lag m+n,x \rag
=B( \de \cos \lag m,x\rag+  b\sin \lag n,x\rag)+ B( b \cos \lag n,x\rag+  \de\sin \lag  m,x\rag).
$$Since $\la\in \R$ is arbitrary, 
from  the definition of $E_1(\WW)$ we get that   $\de\cos\lag m+n,x \rag \in E_1(\WW)$. To prove that $\de\cos\lag m-n,x \rag \in E_1(\WW)$,  it suffices to replace $b$ by $-b$   in \eqref{anhz1}, take the sum of the resulting equality  with \eqref{anh1}:
\begin{align}\label{cru2} \cos\lag m-n,x \rag P_{m-n} \left(\lag a,n\rag b- \lag b,m\rag a \right)&= 
 B(a \cos \lag m,x\rag+ b\sin \lag n,x\rag)\nonumber  \\&\quad +
 B(-b \cos \lag n,x\rag+ a\sin \lag m,x\rag),
\end{align}  and choose $a=\de$ and~$\lag b, m\rag=-\la$
$$
\la  \de\cos\lag m-n,x \rag
=B( \de \cos \lag m,x\rag+  b\sin \lag n,x\rag)+ B(- b \cos \lag n,x\rag+  \de\sin \lag n,x\rag).
$$
 The fact that $\de\sin\lag m\pm n,x \rag \in E_1(\WW)$ is proved in a similar way using the following identities  
 \begin{align*} 
2B(a \cos \lag m,x\rag+ b\cos \lag n,x\rag)&= \sin\lag m-n,x \rag P_{m-n} \left(\lag a,n\rag b- \lag b,m\rag a \right)\nonumber\\&\quad -\sin\lag m+n,x \rag P_{m+n} \left(\lag a,n\rag b+ \lag b,m\rag a \right),
\\
2B(a \sin \lag m,x\rag+ b\sin \lag n,x\rag)&= \sin\lag m-n,x \rag P_{m-n} \left(\lag a,n\rag b- \lag b,m\rag a \right)\nonumber\\&\quad +\sin\lag m+n,x \rag P_{m+n} \left(\lag a,n\rag b+ \lag b,m\rag a \right).
\end{align*}

 {\it Step~2}.  To prove \eqref{U:a10}, let us take any vector  $r\in \WW$  such that  $\EE:=\{m,n,r\}$  is  a linearly independent family\footnote{
 Note that $\EE$ is not necessarily a generator of $\Z^3$.  For example,
   $m=(2,0,0), n=(0,1,0), r=(0,0,1)$ is a basis in $\R^3$, but  not a generator of $\Z^3$,   since the greatest common divisor in   Theorem~\ref{gij} is equal to $2$.   } in~$\R^3$. This choice is possible, by the conditions of the proposition.  For any $\alpha,\beta,\gamma \in\R$, we shall  write $(\alpha,\beta,\gamma)_\EE$  instead of $\alpha m+\beta n+\gamma r$.  Then we have also that the family  $ \{ (1,1,-1)_\EE$,  $(1,-1,1)_\EE$, $  (-1,1,1)_\EE\}$ is independent, hence
$$
\{0\}=   (1,1,-1)_\EE^\bot  \cap   (1,-1,1)_\EE^\bot \cap   (-1,1,1)_\EE^\bot  .
$$
We are going to prove \eqref{U:a10} under the assumption     
   \begin{equation}\label{m+n}
(1,1,1)_\EE\notin  (1,1,-1)_\EE^\bot.
\end{equation}   The other two cases $(1,1,1)_\EE\notin  (1,-1,1)_\EE^\bot$ and $(1,1,1)_\EE\notin  (-1,1,1)_\EE^\bot$ are similar. As  
$$
(1,1,0)_\EE=(1,0,0)_\EE+(0,1,0)_\EE =m+n,
$$ by \eqref{uxxahay}, we  have    
    \begin{equation}\label{U:a13}
  \de(m,n)\cos\lag (1,1,0)_\EE, x\rag\in E_1(\WW).
 \end{equation}  
Writing 
$$(1,1,1)_\EE=(0,0,1)_\EE+(1,1,0)_\EE
$$and applying \eqref{cru1} and \eqref{U:a13}, we obtain  for any $b\in (0,0,1)_\EE^\bot$   that
\begin{align}\label{E:fff} 
&\cos\lag (1,1,1)_\EE,x  \rag P_{(1,1,1)_\EE} \left(\lag  \de (m,n), (0,0,1)_\EE\rag b+ \lag b,(1,1,0)_\EE\rag  \de (m,n) \right)\nonumber   \\&\quad=
B( \de (m,n) \cos \lag (1,1,0)_\EE,x\rag+ b\sin \lag (0,0,1)_\EE,x\rag)\nonumber\\&\quad\quad +  
B(b \cos \lag  (0,0,1)_\EE,x\rag+  \de (m,n)\sin \lag (1,1,0)_\EE,x\rag)\in E_2(\WW).
\end{align} Let us define the  set $$\GG:=\{\lag  \de (m,n), (0,0,1)_\EE\rag b+ \lag b,(1,1,0)_\EE\rag  \de (m,n): b\in(0,0,1)_\EE^\bot \}.$$ Since      $m,n,r$ are linearly independent, we have $\lag  \de (m,n), (0,0,1)_\EE\rag \neq 0$. Thus~$\GG$ is a two-dimensional subspace of $\R^3$ contained in $(1,1,-1)_\EE^\bot$.  This shows  that 
$\GG=(1,1,-1)_\EE^\bot.$ 
 Assumption \eqref{m+n} implies that  
   the orthogonal projection   $P_{(1,1,1)_\EE} \GG$     coincides with $(1,1,1)_\EE^\bot$, so \eqref{E:fff} proves that
   \begin{equation}\label{U:a15}
   \aA_{(1,1,1)_\EE}\subset E_2(\WW).
   \end{equation} Similarly, one can show that $\BB_{(1,1,1)_\EE}\subset E_2(\WW)$. Finally, writing 
$$
(1,1,0)_\EE=(1,1,1)_\EE-(0,0,1)_\EE  
$$and applying \eqref{uxxahay} and \eqref{U:a13} to the set $\WW_1:=\WW\cup\{(1,1,1)_\EE,(0,0,1)_\EE\}$, we see that 
\begin{align*}
 \de ((1,1,1)_\EE,(0,0,1)_\EE)\cos\lag (1,1,0)_\EE, x\rag&\in E_1(\WW_1)=\FF(E(\WW_1))\\&\subset \FF(E_2(\WW))=E_3(\WW).
\end{align*} Combining this with \eqref{U:a13} and the fact that
  $ \de ((1,1,1)_\EE,(0,0,1)_\EE)\nparallel  \de (m,n)$, we get \eqref{U:a10}.  
  This completes the proof of the proposition.  
  \end{proof} 
  \begin{proof}[Proof of Theorem \ref{crit}]   
{\it Step~1}. Let us show that 
\begin{equation}\label{E:generatoA}
E(\Z^3_\KK )\subset E_\ty(\KK).
\end{equation}
 To this end, we    introduce  the sets
 $$
 \KK_0:=\KK, \quad \KK_j=\KK_{j-1}\cup\{m\pm n: m,n\in \KK_{j-1}, m\nparallel n\}, \quad j\ge1.
 $$From Proposition \ref{ppo} it follows that 
 \begin{align}\label{shatkarev}
E(\KK_j) \subset E_3 (\KK_{j-1})&= \FF^3(E(\KK_{j-1}))\subset \FF^{6}(E(\KK_{j-2}))\subset \ldots \subset \FF^{ 3 j}(E(\KK))\nonumber\\&=E_{3j}(\KK).
 \end{align}On the other hand, since $\KK$ is a generator of $\Z^3_\KK$, one   easily checks that $\cup_{j=1}^\ty \KK_j=\Z^3_\KK$. Combining this with \eqref{shatkarev}, we get 
     \eqref{E:generatoA}.

     \smallskip
{\it Step~2}.  Now let us prove that  
\begin{equation}\label{E:generatoB}
 E_\ty(\KK) \subset E(\Z^3_\KK ).
\end{equation} For any $\eta_1\in  E_1(\KK_{j-1})$ and $j\ge1$, there   are vectors $\eta, \zeta^1,\ldots ,\zeta^p\in
E(\KK_{j-1})$   satisfying
the relation
$$
\eta_1=\eta-\sum_{i=1}^p B(\zeta^i).
$$Here we use the following simple lemma. 
\begin{lemma}\label{lemmajan}
For any $j\ge1$, we have
$$
\{B(\zeta): \zeta\in E(\KK_{j-1})\} \subset E(\KK_{j}).
$$
\end{lemma}This lemma implies that
$$
E_1(\KK_{j-1})\subset E(\KK_{j}).
$$Iterating this, we get    
$$
E_j(\KK)\subset E(\KK_{j}),
$$hence 
$$
E_\ty(\KK)=\bigcup_{j=1}^\ty E_j(\KK) \subset  \bigcup_{j=1}^\ty E(\KK_{j}) \subset   E\left( \bigcup_{j=1}^\ty \KK_{j}\right) =E(\Z^3_\KK).
$$This   proves  \eqref{E:generatoB} and \eqref{E:generato}.
     
     \smallskip
{\it Step~3}.  If  $\KK$ is a generator of $\Z^3$, then   \eqref{E:generato} implies  that $E(\KK)$ is saturating in  $H^k_\sigma$ for any $k\ge0$.  

\smallskip

Now let us assume that $\KK$ is not a generator of $\Z^3$, i.e.,  there is $\ell\in \Z^3$   such that $\ell\notin \Z^3_\KK$.  Then it follows from \eqref{E:generato} that $c_\ell$ is orthogonal to $E_\ty(\KK)$ in $H$. This shows that  $E(\KK)$ is not saturating in  $H$ and completes the proof of the theorem.

 \end{proof}

\begin{proof}[Proof of Lemma \ref{lemmajan}] For any $\zeta\in  E(\KK_{j-1})$, we have 
$$
\zeta=\sum_{\ell\in \pm \KK_{j-1}} (a_\ell c_{  \ell} +b_\ell s_{\ell})
$$ for some $a_\ell , b_\ell \in \R$. It follows that 
\begin{align*}
B(\zeta)= &\sum_{m,n\in \pm\KK_{j-1} } ( a_m a_n B(c_m,c_n)+ b_m b_n B(s_m,s_n)\nonumber\\&+a_m b_n B(c_m,s_n)+b_m a_n B(s_m,c_n)).
\end{align*}Using some trigonometric identities, it is easy to verify that 
$$
B(c_m,c_n)\in \textup{span}\{s_{m+n},s_{m-n}\} \subset E(\KK_{j}).
$$ In a similar way, one gets $ B(s_m,s_n),  B(c_m,s_n),  B(s_m,c_n)\in E(\KK_{j}).
$
\end{proof}

  For any finite set $\KK\subset \Z^3$ and $k\ge3$, let us define the space 
  $$
  H^k_{\sigma,\KK}:= \overline {E_\ty(\KK)}^{H^k}.
  $$From   the  structure of the nonlinearity it follows that $ H^k_{\sigma,\KK}$ is invariant for (\ref{2.1}) when 
   $h, \eta\in L^2(J_T,H^{k-1}_{\sigma,\KK})$. Moreover,  $ H^k_{\sigma,\KK}=H^k_{\sigma}$ if and only if    $\KK$ is a generator of $\Z^3$.
 As a corollary   we get the following characterisation of the controllability in $H^k_\sigma$.
\begin{theorem}\label{TTTSD} Let $\KK\subset \Z^3$ be a finite set and $h\in L^2(J_T,H^{k-1}_{\sigma,\KK})$.    Then equation~(\ref{2.1})  is approximately controllable in   the space $ H^k_{\sigma}$ at   time $T$ by controls 
$\eta\in C^\infty(J_T,E(\KK))$  if and only if $\KK$ is a generator of $\Z^3$.
\end{theorem} It is also interesting to study the controllability properties of the NS system when $E(\KK)$ given by \eqref{E:EEE} is not saturating (i.e., $\KK$ is not a generator of~$\Z^3$).  
Let us note that   the space  $E(\KK)$ is saturating in   $H^{k}_{\sigma,\KK}$ for any $\KK\subset \Z^3$ and $k\ge0$ (in the sense that $E_\ty(\KK)$ is dense in  $H^{k}_{\sigma,\KK}$).   We have the following refined version of Theorem~\ref{T.2.1}.
  \begin{theorem}\label{TTT.2.1} For any non-empty  finite $\KK\subset \Z^3$   and    $h\in L^2(J_T,H^{k-1}_{\sigma,\KK})$, equation~(\ref{2.1})  is approximately controllable in  the space $ H^k_{\sigma,\KK}$ at   time~$T$ by controls
$\eta\in C^\infty(J_T,E(\KK))$, i.e.,  for any $\e>0$ and       any $$\varphi\in  C(J_T,H^{k}_{\sigma,\KK})\cap L^2(J_T,H_{\sigma,\KK}^{k+1})\cap W^{1,2}(J_T, H^{k-1}_{\sigma,\KK})$$    there is a control $\eta
\in \Theta(h,u_0)\cap C^\ty(J_T,E(\KK)) $ such that
$$
\|\RR_T(u_0,h+\eta) - \varphi(T) \|_{k}+ \lvvvert  {\RR(u_0,h+\eta)} - \varphi \lvvvert_{ T,  k}+\|\phi^{\RR(u_0,h+\eta) }- \phi^{\varphi} \|_{L^\ty(J_T,C^1)}<\e,
$$where $u_0=\varphi(0)$.   

\end{theorem} The proof of this result literally repeats the arguments of the proof of   Theorem \ref{T.2.1}, so we omit the details.

  \subsubsection{Controls  with
  two vanishing components}
In this section,
we consider the NS system 
 \begin{align}
\p_t u-\nu \Delta u+\lag u, \nabla \rag u+\nabla p  &= h(t,x)+ (0,0,1)\eta (t,x), 
\quad  \diver u =0,   \label{0.zz1}\\
u(0)&=u_0,\label{iczzz}
\end{align} 
 where $\eta$ is a control taking values in a finite-dimensional space of the form
 $$
 \HH(\KK):= \textup{span}\{   \cos \lag m,x\rag,  \sin \lag m,x\rag: m\in\KK  \},  $$  where $\KK$ is a subset of $  \Z^3$, and $h $  is  a given smooth  divergence-free  function.  Let us rewrite~\eqref{0.zz1} in an equivalent form 
 \begin{align}
\dot u-\nu \Delta u+B(u )   &=h(t,x) +   \tilde \eta (t,x),\label{vanish}
\end{align}where $\tilde \eta :=\Pi(e\eta)$ and  $e:=(0,0,1)$. Then 
    the control $\tilde \eta$ takes   values in the space
    \begin{equation}\label{tilde}
   \tilde E (\KK):=\textup{span}\{   (P_{m}e)\cos \lag m,x\rag, (P_m e)\sin \lag m,x\rag: m\in\KK  \}.
 \end{equation}For an appropriate choice of $\KK$, this space is saturating. 
    \begin{proposition}\label{lavt}
  Let 
 \begin{equation}\label{tildeKK}
  \KK:=\{(1,0,0), (0,1,0), (1,0,1), (0,1,1) \}.
 \end{equation}Then $\tilde E (\KK)$ is an $8$-dimensional  saturating  space   in $H^{k}_\sigma$ for any $k\ge0$.
   \end{proposition}   Combining this proposition      with Theorem \ref{T.2.1}, we get immediately the following result.
 \begin{theorem} 
Let $h\in L^2(J_T,H^{k-1}_\sigma), k\ge3$, and $T>0$. If  $\KK$ is defined by~\eqref{tildeKK}, then system \eqref{vanish}  is approximately is controllable  at   time $T$ by controls
$\tilde \eta\in C^\infty(J_T,\tilde E(\KK))$. \end{theorem}
         \begin{proof}  
{\it Step~1}. Let us first show that $\aA_{(0,0,1)}\subset \FF(\tilde E(\KK))$. Using   \eqref{cru2}, we get for any $\la\in \R$  
\begin{align*} 
   \la(-1/2,0 ,0) & \cos \lag (0,0,1),x \rag\\ &  =B(\la(P_{(1,0,0)}e) \cos \lag (1,0,0),x\rag+ (P_{(1,0,1)}e)\sin \lag (1,0,1),x\rag)\\&\quad  +B(-(P_{(1,0,1)}e)\cos \lag (1,0,1),x\rag+ \la(P_{(1,0,0)}e) \sin \lag (1,0,0),x\rag),
  \\  \la (0, -1/2,0) & \cos\lag (0,0,1),x \rag\\\quad\quad&=B(\la(P_{(0,1,0)}e) \cos \lag (0,1,0),x\rag+ (P_{(0,1,1)}e)\sin \lag (0,1,1),x\rag)\\&\quad +B(-(P_{(0,1,1)}e)\cos \lag (0,1,1),x\rag+ \la(P_{(0,1,0)}e) \sin \lag (0,1,0),x\rag) 
.\end{align*} The definition of  $\FF$ implies that   $\aA_{(0,0,1)}\subset \FF(\tilde E(\KK))$. A similar computation gives that        $\BB_{(0,0,1)}\subset \FF(\tilde E(\KK))$.

\smallskip
{\it Step~2}.  Again using \eqref{cru2}, we obtain for any $b:=(b_1,b_2,0)\in \R^3$ 
\begin{align*}  &(0, b_2/2,-b_1/2)\cos\lag (1,0,0),x \rag  = 
 B((P_{(1,0,1)}e) \cos \lag (1,0,1),x\rag+ b\sin \lag (0,0,1),x\rag)\nonumber  \\&\quad +
 B(-b \cos \lag (0,0,1),x\rag+ (P_{(1,0,1)}e)\sin \lag (1,0,1),x\rag) \in\FF^2(\tilde E(\KK)).
\end{align*}  This shows that $\aA_{ (1,0,0)}\subset \FF^2(\tilde E(\KK))$. Similarly one  proves  also $$\BB_{ (1,0,0)}, \aA_{(0,1,0)}, \BB_ {(0,1,0)}\subset \FF^2(\tilde E(\KK)).$$ Thus the result follows from the fact that $\{(1,0,0), (0,1,0), (0,0,1)\}$ is a generator of $\Z^3$.

\end{proof}

 \subsubsection{6-dimensional example} \label{S:2.2.3}

The following result, combined with Theorem \ref{T.2.1}, shows that that the 3D NS system can be  approximately controlled with~$\eta$ taking values    in a    $6$-dimensional space.
    \begin{proposition}\label{lsdfavt}
 Let us define the  following  $6$-dimensional space 
  \begin{align}\label{tilde}
   \hat  E :=\textup{span}\{ &  a\cos \lag (1,0,1),x\rag, a \sin \lag (1,0,1),x\rag,  \nonumber\\& e\cos \lag (0,1,1),x\rag, e \sin \lag (0,1,1),x\rag, \nonumber\\& b\cos \lag (0,0,1),x\rag, b \sin \lag (0,0,1),x\rag \},
 \end{align}  where $a:=(1,1,1), b:=(1,0,0), e:=(0,0,1)$.
 Then $\hat  E $ is   saturating     in $H^{k}_\sigma$ for any~$k\ge0$.
   \end{proposition} 
   \begin{proof}  
{\it Step~1}. Let us first show that $\aA_{(1,-1,0)}\subset \FF^2(\hat  E)$. Using   \eqref{cru2}, we get for any $\la\in \R$  
\begin{align*} 
 \la (0, -1,-1) &\cos \lag (1,0,0),x \rag=B(\la a \cos \lag (1,0,1),x\rag+ b\sin \lag (0,0,1),x\rag)\\&  +B(-b\cos \lag (0,0,1),x\rag+ \la a \sin \lag (1,0,1),x\rag)\in \FF(\hat  E),\\
 \la (1,0,0) &\cos \lag (0,1,0),x \rag=B( \la e \cos \lag (0,1,1),x\rag+ b\sin \lag (0,0,1),x\rag)\\&  +B(-b\cos \lag (0,0,1),x\rag+ \la e \sin \lag (0,1,1),x\rag)\in \FF(\hat  E)
\end{align*}   and $ (0, -1,-1)  \sin \lag (1,0,0),x \rag, (1,0,0)  \sin \lag (0,1,0),x \rag\in \FF(\hat  E)$, similarly.
Writing
$$
(1,-1,0)=(1,0,0)-(0,1,0)=(1,0,1)-(0,1,1)
$$and applying \eqref{cru2}, we see that 
\begin{align*} 
  \la(0,0,1) &\cos\lag (1,-1,0),x \rag=B(  \la(0, -1,-1) \cos \lag (1,0,0),x\rag+b\sin \lag (0,1,0),x\rag)\\&  +B(-b\cos \lag (0,1,0),x\rag+ \la (0, -1,-1)\sin \lag (1,0,0),x\rag)\in \FF^2(\hat  E),\\
\la  (-1,-1,1) &\cos  \lag (1,-1,0),x \rag=B(\la a \cos \lag (1,0,1),x\rag+ e\sin \lag (0,1,1),x\rag)\\&  +B(-e\cos \lag (0,1,1),x\rag+\la a\sin \lag (1,0,1),x\rag)\in \FF^2(\hat  E)
.\end{align*}  This proves that  $\aA_{(1,-1,0)}\subset \FF^2(\hat  E)$. A similar computation establishes  that   $\BB_{(1,-1,0)}\subset \FF^2(\hat  E)$.

\smallskip
{\it Step~2}.  Let us show that $\aA_{(1,0,0)}, \BB_{(1,0,0)}\subset \FF^3(\hat  E) $. Taking  any  vector  $f:=(f_1,f_1,f_2) \in (1,-1,0)^\bot$, we  apply   \eqref{cru1}   
\begin{align*} 
  (0, f_1,f_2) \cos\lag &(1,0,0),x \rag=B( f \cos \lag (1,-1,0),x\rag+ b\sin \lag (0,1,0),x\rag)\\&\quad +B(b\cos \lag (0,1,0),x\rag+ f \sin \lag (1,-1,0),x\rag)\in \FF^3(\hat  E)
.\end{align*}  This proves that $\aA_{(1,0,0)}\subset \FF^3(\hat  E) $, and $\BB_{(1,0,0)}\subset \FF^3(\hat  E)$  is similar.

\smallskip
{\it Step~3}.  Let us show that $\aA_{(0,0,1)} , \BB_{(0,0,1)}\subset \FF^4(\hat  E)  $.  Again we shall prove only the first inclusion. 
For any  $g:=(0,g_1,g_2) \in \R^3$, we  apply   \eqref{cru2}   
\begin{align*} 
  (-g_2,g_1-g_2,0) \cos\lag &(0,0,1),x \rag=B(a \cos \lag (1,0,1),x\rag+ g\sin \lag (1,0,0),x\rag)\\&\quad +B(-g\cos \lag (1,0,0),x\rag+ a \sin \lag (1,0,1),x\rag)\in \FF^3(\hat  E)
.\end{align*}  This proves that $\aA_{(0,0,1)}\subset \FF^4(\hat  E) $ and $\BB_{(0,0,1)}\subset \FF^4(\hat  E)$ is similar.  By Theorem~\ref{gij}, we have that  the family   $\{(1,0,0), (0,0,1), (1,-1,0)\}$ is a generator of~$\Z^3$. Thus applying Theorem \ref{crit}, we complete the proof.

\end{proof}
 
 It would be interesting to get  a   characterisation of   finite-dimensional saturating spaces of the following general  form  
 $$
 E(\KK_c,\KK_s, a,b):=\textup{span}\{ a_m\cos\lag m,x\rag; b_n\sin\lag n,x\rag: m\in \KK_c, n\in\KK_s \},
 $$
 where $\KK_c,\KK_s\subset \Z^3$, $a:=\{a_m\}_{m\in \KK_c}\subset \R^3_*$, and $b:=\{b_n\}_{n\in \KK_s}\subset \R^3_*$. From the results of Subsection \ref{type1} it follows that both~$\KK_c$ and  $\KK_s$ are    necessarily       generators of  $\Z^3$.

    \addcontentsline{toc}{section}{Bibliography}
\def\cprime{$'$} \def\cprime{$'$}
  \def\polhk#1{\setbox0=\hbox{#1}{\ooalign{\hidewidth
  \lower1.5ex\hbox{`}\hidewidth\crcr\unhbox0}}}
  \def\polhk#1{\setbox0=\hbox{#1}{\ooalign{\hidewidth
  \lower1.5ex\hbox{`}\hidewidth\crcr\unhbox0}}}
  \def\polhk#1{\setbox0=\hbox{#1}{\ooalign{\hidewidth
  \lower1.5ex\hbox{`}\hidewidth\crcr\unhbox0}}} \def\cprime{$'$}
  \def\polhk#1{\setbox0=\hbox{#1}{\ooalign{\hidewidth
  \lower1.5ex\hbox{`}\hidewidth\crcr\unhbox0}}} \def\cprime{$'$}
  \def\cprime{$'$} \def\cprime{$'$} \def\cprime{$'$}
\providecommand{\bysame}{\leavevmode\hbox to3em{\hrulefill}\thinspace}
\providecommand{\MR}{\relax\ifhmode\unskip\space\fi MR }
\providecommand{\MRhref}[2]{%
  \href{http://www.ams.org/mathscinet-getitem?mr=#1}{#2}
}
\providecommand{\href}[2]{#2}

\end{document}